\documentclass[11pt, reqno]{article}

\usepackage{amsmath,amssymb,amsfonts,amsthm}
\usepackage{color}

\usepackage[colorlinks=true,linkcolor=blue,citecolor=blue,urlcolor=blue,pdfborder={0 0 0}]{hyperref}
\usepackage{cleveref}

\usepackage{caption}
\usepackage{thmtools}

\usepackage{mathrsfs}

\crefname{theorem}{Theorem}{Theorems}
\crefname{thm}{Theorem}{Theorems}
\crefname{lemma}{Lemma}{Lemmas}
\crefname{lem}{Lemma}{Lemmas}
\crefname{remark}{Remark}{Remarks}
\crefname{prop}{Proposition}{Propositions}
\crefname{defn}{Definition}{Definitions}
\crefname{corollary}{Corollary}{Corollaries}
\crefname{conjecture}{Conjecture}{Conjectures}
\crefname{question}{Question}{Questions}
\crefname{chapter}{Chapter}{Chapters}
\crefname{section}{Section}{Sections}
\crefname{figure}{Figure}{Figures}
\crefname{example}{Example}{Examples}

\theoremstyle{plain}
\newtheorem{thm}{Theorem}[section]

\newtheorem{lemma}[thm]{Lemma}
\newtheorem{theorem}[thm]{Theorem}

\newtheorem{corollary}[thm]{Corollary}

\newtheorem{prop}[thm]{Proposition}

\newtheorem{question}[thm]{Question}
\theoremstyle{definition}

\newtheorem{problem}[thm]{Problem}

\theoremstyle{remark}
\newtheorem{remark}[thm]{Remark}

\numberwithin{equation}{section}

\renewcommand{\P}{\mathbb P}

\newcommand{\R}{\mathbb R}
\newcommand{\Z}{\mathbb Z}
\newcommand{\N}{\mathbb N}

\newcommand{\cB}{\mathcal B}

\newcommand{\cF}{\mathcal F}

\newcommand{\sA}{\mathscr A}
\newcommand{\sB}{\mathscr B}
\newcommand{\sC}{\mathscr C}
\newcommand{\sD}{\mathscr D}

\newcommand{\bbH}{\mathbb H}

\usepackage[margin=1in]{geometry}
\usepackage{tikz}
\usepackage{pgfplots}
\usepackage{extarrows}
\usepackage{titling}
\usepackage{commath}
\usepackage{wasysym}
\usepackage{mathtools}
\usepackage{titlesec}
\newcommand{\eps}{\varepsilon}
\usepackage{bbm}
\usepackage{setspace}
\usepackage{enumitem}
\usepackage{booktabs}

\usepackage[textsize=tiny]{todonotes}

\usepackage{cite}

\newcommand{\bP}{\mathbf P}
\newcommand{\bE}{\mathbf E}

\newcommand{\opleq}{\preccurlyeq}

\def\P{\mathbb{P}}

\newcommand{\sS}{\mathscr{S}}

\DeclareMathSymbol{\leqslant}{\mathalpha}{AMSa}{"36} 
\DeclareMathSymbol{\geqslant}{\mathalpha}{AMSa}{"3E} 
\DeclareMathSymbol{\eset}{\mathalpha}{AMSb}{"3F}     

\renewcommand{\epsilon}{\varepsilon}

\makeatletter
\pgfdeclareshape{slash underlined}
{
  \inheritsavedanchors[from=rectangle] 
  \inheritanchorborder[from=rectangle]
  \inheritanchor[from=rectangle]{north}
  \inheritanchor[from=rectangle]{north west}
  \inheritanchor[from=rectangle]{north east}
  \inheritanchor[from=rectangle]{center}
  \inheritanchor[from=rectangle]{west}
  \inheritanchor[from=rectangle]{east}
  \inheritanchor[from=rectangle]{mid}
  \inheritanchor[from=rectangle]{mid west}
  \inheritanchor[from=rectangle]{mid east}
  \inheritanchor[from=rectangle]{base}
  \inheritanchor[from=rectangle]{base west}
  \inheritanchor[from=rectangle]{base east}
  \inheritanchor[from=rectangle]{south}
  \inheritanchor[from=rectangle]{south west}
  \inheritanchor[from=rectangle]{south east}
  \inheritanchorborder[from=rectangle]
  \foregroundpath{
    \southwest \pgf@xa=\pgf@x \pgf@ya=\pgf@y
    \northeast \pgf@xb=\pgf@x \pgf@yb=\pgf@y
    \pgf@xc=\pgf@xa
    \advance\pgf@xc by .5\pgf@xb
    \pgf@yc=\pgf@ya
    \advance\pgf@xc by -1.3pt
    \advance\pgf@yc by -1.8pt
    \pgfpathmoveto{\pgfqpoint{\pgf@xc}{\pgf@yc}}
    \advance\pgf@xc by  2.6pt
    \advance\pgf@yc by  3.6pt
    \pgfpathlineto{\pgfqpoint{\pgf@xc}{\pgf@yc}}
    \pgfpathmoveto{\pgfqpoint{\pgf@xa}{\pgf@ya}}
    \pgfpathlineto{\pgfqpoint{\pgf@xb}{\pgf@ya}}
 }
}
\tikzset{nomorepostaction/.code=\let\tikz@postactions\pgfutil@empty}

\makeatother

\makeatletter
\DeclareFontFamily{OMX}{MnSymbolE}{}
\DeclareSymbolFont{MnLargeSymbols}{OMX}{MnSymbolE}{m}{n}
\SetSymbolFont{MnLargeSymbols}{bold}{OMX}{MnSymbolE}{b}{n}
\DeclareFontShape{OMX}{MnSymbolE}{m}{n}{
    <-6>  MnSymbolE5
   <6-7>  MnSymbolE6
   <7-8>  MnSymbolE7
   <8-9>  MnSymbolE8
   <9-10> MnSymbolE9
  <10-12> MnSymbolE10
  <12->   MnSymbolE12
}{}
\DeclareFontShape{OMX}{MnSymbolE}{b}{n}{
    <-6>  MnSymbolE-Bold5
   <6-7>  MnSymbolE-Bold6
   <7-8>  MnSymbolE-Bold7
   <8-9>  MnSymbolE-Bold8
   <9-10> MnSymbolE-Bold9
  <10-12> MnSymbolE-Bold10
  <12->   MnSymbolE-Bold12
}{}

\let\llangle\@undefined
\let\rrangle\@undefined
\DeclareMathDelimiter{\llangle}{\mathopen}%
                     {MnLargeSymbols}{'164}{MnLargeSymbols}{'164}
\DeclareMathDelimiter{\rrangle}{\mathclose}%
                     {MnLargeSymbols}{'171}{MnLargeSymbols}{'171}
\makeatother

\pretitle{\begin{flushleft}\Large}
\posttitle{\par\end{flushleft}\vskip 0.5em
}
\preauthor{\begin{flushleft}}
\postauthor{
\\
\vspace{0.5em}
\footnotesize{The Division of Physics, Mathematics and Astronomy, California Institute of Technology\\ Email: 
\href{mailto:t.hutchcroft@caltech.edu}{t.hutchcroft@caltech.edu}}
\end{flushleft}}
\predate{\begin{flushleft}}
\postdate{\par\end{flushleft}}
\title{\bf Sharp hierarchical upper bounds on the critical two-point function for long-range percolation on $\Z^d$}

\renewenvironment{abstract}
 {\par\noindent\textbf{\abstractname.}\ \ignorespaces}
 {\par\medskip}

\titleformat{\section}
  {\normalfont\large \bf}{\thesection}{0.4em}{}

\author{{\bf Tom Hutchcroft}}

\begin{document}

\date{\small{\today}}

\maketitle

\setstretch{1.1}

\begin{abstract}
Consider long-range Bernoulli percolation on $\Z^d$ in which we connect each pair of distinct points $x$ and $y$ by an edge with probability  $1-\exp(-\beta\|x-y\|^{-d-\alpha})$, where $\alpha>0$ is fixed and $\beta\geq 0$ is a parameter. We prove that if $0<\alpha<d$ then the critical two-point function satisfies 
\[
\frac{1}{|\Lambda_r|}\sum_{x\in  \Lambda_r} \bP_{\beta_c}(0\leftrightarrow x) \preceq r^{-d+\alpha}
\]
for every $r\geq 1$, where $\Lambda_r=[-r,r]^d \cap \Z^d$. In other words, the critical two-point function on $\Z^d$ is always bounded above on average by the critical two-point function on the hierarchical lattice. This upper bound is believed to be sharp for values of $\alpha$ strictly below the \emph{crossover value} $\alpha_c(d)$, where the values of several critical exponents for long-range percolation on $\Z^d$ and the hierarchical lattice are believed to be equal.



\end{abstract}

\section{Introduction}
\label{sec:intro}

Let $d\geq 1$ and let $J:\Z^d \times \Z^d \to [0,\infty)$ be a kernel that is \textbf{symmetric} in the sense that $J(x,y)=J(y,x)$ for every $x,y\in \Z^d$. For each $\beta\geq 0$, \textbf{long-range percolation} on $\Z^d$ with kernel $J$ is the random graph with vertex set $\Z^d$ in which each potential edge $\{x,y\}$ is included independently at random with probability $1-\exp(-\beta J(x,y))$. We write $\bP_\beta=\bP_{\beta,J}$ for the law of the resulting random graph. We will be particularly interested in the case that $J$ is \textbf{translation-invariant}, meaning that $J(x,y)=J(0,y-x)$ for every $x,y\in \Z^d$ and \textbf{integrable}, meaning that $\sum_{y\in \Z^d} J(x,y)<\infty$ for every $x$. For many purposes, the most interesting case (besides nearest-neighbour models) occurs when $J(x,y)$ decays like an inverse power of $\|x-y\|$, so that
\begin{equation}
J(x,y) \sim A \|x-y\|^{-d-\alpha} \qquad \text{ as $x-y\to\infty$}
\end{equation}
for some constants $A>0$ and $\alpha>0$; smaller values of $\alpha$ make longer edges more likely. Percolation theorists are particularly interested in the geometry of the \textbf{clusters} (connected components) of this random graph and how this geometry changes as the parameter $\beta$ is varied. Indeed, as with nearest-neighbour percolation, much of the interest of the model stems from the fact that it typically undergoes a \emph{phase transition}, in which an infinite cluster emerges as $\beta$ is varied through the \textbf{critical value}
\[
\beta_c=\beta_c(J):=\inf\{\beta \geq 0: \bP_\beta(\text{an infinite cluster exists})>0 \},
\]
which satisfies $0<\beta_c<\infty$ if $d\geq 2$ and $\alpha>0$ or $d=1$ and $0<\alpha \leq 1$ \cite{schulman1983long,newman1986one}. In this paper we  study the behaviour of the model \emph{at criticality} (i.e., when $\beta=\beta_c$), where one expects a rich, fractal-like geometry to emerge \cite{ding2013distances,MR3306002,MR2430773,duminil2020long,hutchcroft2020power}; see \cite{heydenreich2015progress} and the very recent paper \cite{biskup2021arithmetic} for detailed literature reviews regarding other aspects of the model.

\medskip

Perhaps surprisingly, long-range percolation is understood rather better at criticality than nearest-neighbour percolation, at least for small values of $\alpha$. Indeed, Noam Berger \cite{MR1896880} proved in 2002  that the phase transition is continuous in the sense that there are no infinite clusters at $\beta_c$ whenever $0<\alpha <d$, while the analogous statement for nearest-neighbour percolation in dimensions $3\leq d\leq 6$ is a notorious open problem. Although it is believed that the phase transition should be continuous for all $\alpha>0$ when $d\geq 2$, Berger's result is best possible in general since the model has a \emph{discontinuous} phase transition when $d=\alpha=1$ \cite{MR868738,duminil2020long}.

\medskip

\textbf{Critical exponents.}
Once one knows that the phase transition is continuous, it becomes a question of central interest to understand the \emph{critical exponents} associated to the model, which are believed to describe the large-scale geometry of the model at and near criticality \cite[Chapters 9 and 10]{grimmett2010percolation}. 
We will focus on the exponents traditionally denoted by $\delta$ and $\eta$ which, if they exist, are defined to satisfy
\begin{align*}
\bP_{\beta_c}(|K| \geq n) &\approx n^{-1/\delta} &\text{ as }n&\to\infty\\
\hspace{-2cm}\text{and} \hspace{2cm} \bP_{\beta_c}(x\leftrightarrow y) &\approx \|x-y\|^{-d+2-\eta} &\text{ as }\|x-y\|&\to\infty,
\end{align*}
where $\approx$ means that the ratio of the logarithms of the two sides tends to $1$ in the relevant limit, $K=K(0)$ denotes the cluster of the origin, and $\{x\leftrightarrow y\}$ denotes the event that $x$ and $y$ are connected (i.e., belong to the same cluster). We will often refer to the connection probability $\bP_{\beta_c}(x\leftrightarrow y)$ as the \textbf{two-point function}.
 These exponents are expected to depend on the dimension $d$ and the long-range parameter $\alpha$ but not on the small-scale details of the model such as the precise choice of kernel $J$. Computing and/or proving the existence of critical exponents is typically a very challenging problem that is of central importance throughout mathematical physics. For percolation, progress has been limited mostly to the 
\emph{high-dimensional} case ($d>6$ or $\alpha<d/3$), where the lace expansion \cite{MR2239599,hara1994mean,heydenreich2015progress} has been developed as a powerful and general method for proving that these exponents take their \emph{mean-field} values $\delta=2$ and $\eta=0 \vee (2-\alpha)$ \cite{MR1043524,MR762034,MR1127713,fitzner2015nearest,MR3306002,MR1959796,MR2430773}, and to the 
  (very special) case of site percolation on the triangular lattice where the theory of conformally-invariant processes applies \cite{MR879034,smirnov2001critical,smirnov2001critical2,lawler2002one} and these exponents are known to be $\delta=91/5$ and $\eta=5/24$ as predicted by Nienhuis \cite{nienhuis1987coulomb}. There are no conjectured exact values for these exponents for nearest-neighbour percolation in dimensions $3$, $4$, or $5$, in which case their values are likely to be transcendental. 

\medskip

While critical exponents in low dimensions remain rather mysterious quantities in general, there is a surprisingly simple prediction for the dependence of these exponents on the long-range parameter $\alpha$ when $d$ is fixed: It is believed that if $\eta_\mathrm{SR}=\eta_\mathrm{SR}(d)$ denotes the analogous critical exponent for nearest-neighbour percolation then 
\begin{equation}
\label{eq:eta_prediction}
2-\eta = \begin{cases} \alpha & \alpha \leq \alpha_c\\
2-\eta_\mathrm{SR} & \alpha>\alpha_c,
\end{cases}
\end{equation}
where the \emph{crossover value} $\alpha_c=2-\eta_\mathrm{SR}$ is the unique value making this function continuous. In particular, when $d<6$ the exponent $2-\eta$ is expected to `stick' to its mean-field value\footnote{The fact that $2-\eta=\alpha \wedge 2$ is the mean-field value of this exponent is related to the fact that the inverse of the fractional Laplacian $(-\Delta)^{\alpha/2}$ decays as $\|x\|^{-d+\alpha}$ for $\alpha<2$ and as $\|x\|^{-d+2}$ for $\alpha>2$ \cite[Sections 2 and 3]{MR3772040}.} of $\alpha \wedge 2$ when $\alpha$ belongs to the interval $[d/3,\alpha_c \wedge 2]$, even though the other exponents (such as $\delta$, see \eqref{eq:delta_prediction}) are not expected to take their mean-field values in this interval. 
This prediction first arose in the 1973 work of Sak \cite{sak1973recursion} in the context of the Ising model, with similar predictions for percolation discussed in various places in the physics literature \cite{brezin2014crossover,behan2017scaling,luijten1997interaction}; a detailed numerical study of these predictions for one-dimensional long-range percolation is given in \cite{gori2017one}.
We refer the reader to \cite{MR3306002,MR4032873,MR2430773} for rigorous proofs in certain high-dimensional cases and to \cite{MR3772040,MR3723429} for related results for the long-range spin $O(n)$ model.
Applying the (conjectural) scaling and hyperscaling relations relating $\eta$ and $\delta$ outside the mean-field regime, \eqref{eq:eta_prediction} leads to the prediction
\begin{equation}
\label{eq:delta_prediction}
\delta = 
\begin{cases}
2 & \hspace{1.855em}0< \alpha \leq d/3\\
(d+\alpha)/(d-\alpha) & \hspace{0.84em}d/3 \leq \alpha \leq \alpha_c\\
\delta_\mathrm{SR} & \hspace{1.325em}\alpha_c \leq \alpha <\infty,
\end{cases}
\end{equation}
where $\delta_\mathrm{SR}$ is the analogous critical exponent for the nearest-neighbour model. Unfortunately, a complete proof of either \eqref{eq:eta_prediction} or \eqref{eq:delta_prediction} seems well beyond the scope of existing methods in low dimensions. 

\pgfplotsset{width=0.375\textwidth}
\pgfplotsset{compat=newest}

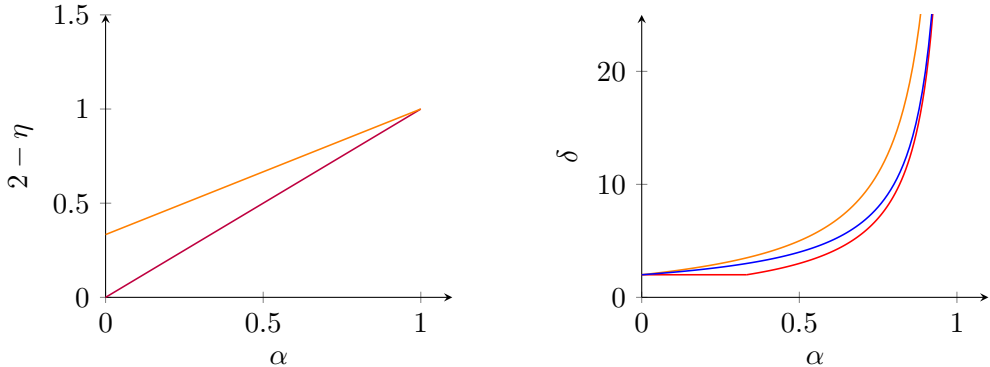
\begin{figure}[t]
\centering
\begin{tikzpicture}
\begin{axis}[
    axis lines = left,
    xlabel = $\alpha$,
    ylabel = {$2-\eta$},
    xmin=0,xmax=1.1,
    ymax=1.5
]
\addplot [
    domain=0:1, 
    samples=100, 
    color=purple,
    semithick
]
{x};
\addplot [
    domain=0:1, 
    samples=100, 
    color=orange,
    semithick
    ]
    {(1/3)+(2*x)/3};
\end{axis}
\end{tikzpicture}
    \hspace{1cm}
    \begin{tikzpicture}
\begin{axis}[
    axis lines = left,
    xlabel = $\alpha$,
    ylabel = {$\delta$},
    xmin=0, xmax=1.1,
    ymin=0, ymax=25
]
\addplot [
    domain=0:1/3, 
    samples=100, 
    color=red,
    semithick
]
{2};
\addplot [
    domain=1/3:0.99, 
    samples=100, 
    color=red,
    semithick
]
{(1+x)/(1-x)};
\addplot [
    domain=0:0.99, 
    samples=100, 
    color=orange,
    semithick
    ]
    {(2+x)/(1-x)};
    \addplot [
    domain=0:0.99, 
    samples=100, 
    color=blue,
    semithick
    ]
    {2/(1-x)};
\end{axis}
\end{tikzpicture}
    \caption{(Exponent estimates for $d=1$.) Our new upper bounds (blue) vs.\ the conjectured true values (red) and the upper bounds proven in \cite{hutchcroft2020power} (orange) of $2-\eta$ and $\delta$ when $d=1$. The part of the graph where our upper bound coincides exactly with the conjectured true value is represented in purple (in fact this is the entire graph of $2-\eta$ in this case). For $d=1$ and $\alpha>1$ there is no phase transition and the exponents are not defined.}
    \label{fig:1d}
    \vspace{-1em}
\end{figure}

\pgfplotsset{width=0.32\textwidth}
\pgfplotsset{compat=newest}

\begin{figure}[t]
\centering
\begin{tikzpicture}
\begin{axis}[
    axis lines = left,
    xlabel = $\alpha$,
    ylabel = {$2-\eta$},
    xmin=0,xmax=3,
    ymax=2.5
]
\addplot [
    domain=0:43/24, 
    samples=100, 
    color=purple,
    semithick
]
{x};
\addplot [
    domain=43/24:3, 
    samples=100, 
    color=red,
    semithick
]
{43/24};
\addplot [
    domain=0:2, 
    samples=100, 
    color=orange,
    semithick
    ]
    {(2/3)+(2*x)/3};
    \addplot [
    domain=43/24:2, 
    samples=100, 
    color=blue,
    semithick
    ]
    {x};
\end{axis}
\end{tikzpicture}
    \hspace{0.25cm}
    \begin{tikzpicture}
\begin{axis}[
    axis lines = left,
    xlabel = $\alpha$,
    ylabel = {$\delta$},
    xmin=0, xmax=3,
    ymin=0, ymax=25
]
\addplot [
    domain=0:2/3, 
    samples=100, 
    color=red,
    semithick
]
{2};
\addplot [
    domain=2/3:43/24, 
    samples=100, 
    color=red,
    semithick
]
{(2+x)/(2-x)};
\addplot [
    domain=43/24:3, 
    samples=100, 
    color=red,
    semithick
]
{91/5};
\addplot [
    domain=0:1.8, 
    samples=100, 
    color=orange,
    semithick
    ]
    {(4+x)/(2-x)};
    \addplot [
    domain=0:1.9, 
    samples=100, 
    color=blue,
    semithick
    ]
    {4/(2-x)};
\end{axis}
\end{tikzpicture}
   \hspace{0.25cm}
    \begin{tikzpicture}
\begin{axis}[
    axis lines = left,
    xlabel = $\alpha$,
    ylabel = {$\delta$},
    xmin=1.7, xmax=1.9,
    ymin=14, ymax=23
]
\addplot [
    domain=1.7:43/24, 
    samples=100, 
    color=red,
    semithick
]
{(2+x)/(2-x)};
\addplot [
    domain=43/24:1.9, 
    samples=100, 
    color=red,
    semithick
]
{91/5};
\addplot [
    domain=1.7:1.8, 
    samples=100, 
    color=orange,
    semithick
    ]
    {(4+x)/(2-x)};
    \addplot [
    domain=1.7:1.9, 
    samples=100, 
    color=blue,
    semithick
    ]
    {4/(2-x)};
\end{axis}
\end{tikzpicture}
    \caption{(Exponent estimates for $d=2$.) Our new upper bounds (blue) vs.\ the conjectured true values (red) and the upper bounds proven in \cite{hutchcroft2020power} (orange) of $2-\eta$ and $\delta$ when $d=2$. The part of the graph where our upper bound coincides exactly with the conjectured true value is represented in purple. Our upper bound on $\delta$ is within $6\%$ of the conjectured true value for $d=2$ and $\alpha=\alpha_c=43/24$; the rightmost plot is a zoomed-in copy of the middle plot around this value.}
    \label{fig:2d}
    \vspace{-1em}
\end{figure}
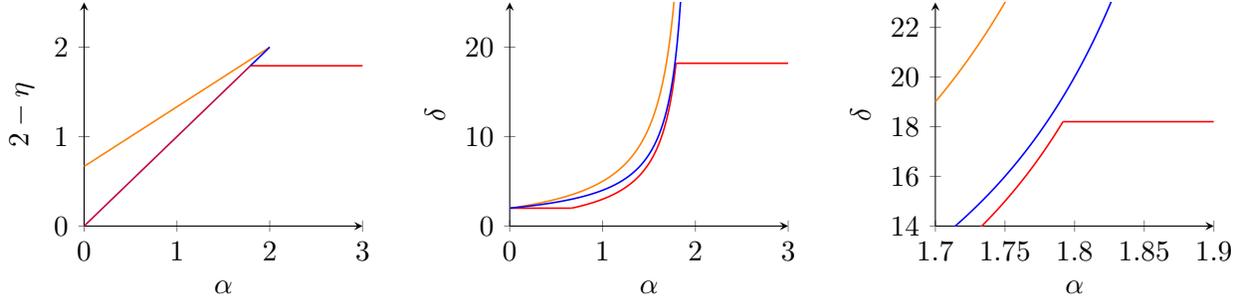

In our recent work \cite{hutchcrofthierarchical}, we proved up-to-constants estimates on the critical two-point function for long-range percolation on the \emph{hierarchical lattice} implying that the critical exponent $\eta$ always satisfies $2-\eta=\alpha$ in this case. (It remains open to compute various other exponents including $\delta$ in this setting.)
In light of this work, the prediction \eqref{eq:eta_prediction} has the following interpretation: For $\alpha<\alpha_c$ long-range effects dominate and the long-range Euclidean model has the same exponents\footnote{Here we are avoiding making the stronger statement that the models \emph{belong to the same universality class} since this claim would arguably be too strong. Indeed, while the models may share exponents, they should not have a common scaling limit as in the Euclidean case relevant continuum limit should be defined on $\R^d$ while in the hierarchical case it should be defined on the $L$-adic numbers. See \cite{LimitsOfUniversality} for detailed discussions of related phenomena. It is also unclear at present whether one should expect the exponents describing off-critical behaviour to coincide.} as the long-range hierarchical model with the same parameter, while for $\alpha>\alpha_c$ short-range effects dominate and the long-range Euclidean model has the same exponents as the nearest-neighbour Euclidean model. In particular, we find it helpful to think of the value $\eta=2-\alpha$ taken below $\alpha<\alpha_c$ as being the \emph{hierarchical value} of $\eta$ rather than the mean-field value of $\eta$ per se. When $\alpha=\alpha_c$ the two effects are comparable and logarithmic corrections to scaling are expected to be present. Rigorous results in high dimensions have been obtained via a lace expansion analysis by Chen and Sakai \cite{MR4032873,sakai2018crossover,MR3306002}, who prove in particular that if $d>6$ and $J$ satisfies a certain perturbative criterion (i.e., is sufficiently `spread-out') then
\begin{equation}
\bP_{\beta_c}(0 \leftrightarrow x) \asymp \begin{cases} 
\|x\|^{-d+\alpha} & \alpha<2\\
 \|x\|^{-d+2} \frac{1}{\log \|x\|} & \alpha = 2\\
 \|x\|^{-d+2} & \alpha > 2
\end{cases}
\qquad \text{ as $\|x\|\to\infty$.}
\end{equation}
 When $d=6$ and $\alpha=2$ there should be two competing sources of logarithmic corrections 
from being both at the upper-critical dimension and at the crossover value.
Note that for $d=1$ we expect the Euclidean and hierarchical models to have the same exponents for all $0<\alpha<1$, while for $\alpha=1$ the Euclidean model has a discontinuous phase transition and the hierarchical model has no phase transition at all. 

\medskip



\textbf{Our results.}
In our previous work \cite{hutchcroft2020power} we made a modest first step towards the understanding of the problem by proving power-law upper bounds on the two-point function and cluster volume tail for long-range percolation with $0<\alpha<d$, but with exponents strictly larger than those predicted by \eqref{eq:eta_prediction} and \eqref{eq:delta_prediction}. In this paper we significantly improve upon this result by proving an upper bound on the two-point function which matches the conjectured true behaviour for $\alpha$ below the crossover value $\alpha_c$. 
We write $\Lambda_r=[-r,r]\cap \Z^d$ for each $r\geq 1$ and write $\|\cdot\|=\|\cdot\|_\infty$.

\begin{theorem}
\label{thm:main}
Let $d\geq 1$, let $J:\Z^d\times \Z^d \to [0,\infty)$ be a symmetric, integrable, translation-invariant kernel, and suppose that there exist constants $0<\alpha<d$ and $c>0$ such that $J(x,y) \geq c\|x-y\|^{-d-\alpha}$ for all distinct $x,y \in \Z^d$. Then there exists a constant $A=A(d,\alpha)$ such that
\[
\frac{1}{|\Lambda_r|}\sum_{x\in  \Lambda_r} \P_{\beta_c}(0\leftrightarrow x) \leq \frac{A}{c\beta_c} r^{-d+\alpha}
\] 
for every $r\geq 1$. In particular, the exponent $\eta$ satisfies $2-\eta \leq \alpha$ if it is well-defined.
\end{theorem}

\begin{remark}
The proof of this theorem is effective in the sense that it gives an explicit (but fairly large) estimate on the constant $A$. Indeed, we believe that this constant can be taken of the form $C^{(d+\alpha)/(d-\alpha)}$ for a universal constant $C$ of order around $10^{50}$. To simplify the exposition we do not keep careful track of the constants arising in our proofs, which we have not attempted to optimize. It may be possible to get a constant of reasonable order with further work.
\end{remark}

\pgfplotsset{width=0.375\textwidth}
\pgfplotsset{compat=newest}

\begin{figure}
\centering
\begin{tikzpicture}
\begin{axis}[
    axis lines = left,
    xlabel = $\alpha$,
    ylabel = {$2-\eta$},
    ymax=3.5
]
\addplot [
    domain=0:2.0457, 
    samples=100, 
    color=purple,
    semithick
]
{x};
\addplot [
    domain=2.0457:4, 
    samples=100, 
    color=red,
    semithick
]
{2.0457};
\addplot [
    domain=0:3, 
    samples=100, 
    color=orange,
    semithick
    ]
    {1+(2*x)/3};
    \addplot [
    domain=2.0457:3, 
    samples=100, 
    color=blue,
    semithick
    ]
    {x};
\end{axis}
\end{tikzpicture}
    \hspace{1cm}
        \begin{tikzpicture}
\begin{axis}[
    axis lines = left,
    xlabel = $\alpha$,
    ylabel = {$\delta$},
    xmin=0, xmax=4,
    ymin=0, ymax=10
]
\addplot [
    domain=0:1, 
    samples=100, 
    color=red,
    semithick
]
{2};
\addplot [
    domain=1:2.0457, 
    samples=100, 
    color=red,
    semithick
]
{(3+x)/(3-x)};
\addplot [
    domain=2.0457:4, 
    samples=100, 
    color=red,
    semithick
]
{5.2886};
\addplot [
    domain=0:2.5, 
    samples=100, 
    color=orange,
    semithick
    ]
    {(6+x)/(3-x)};
    \addplot [
    domain=0:2.5, 
    samples=100, 
    color=blue,
    semithick
    ]
    {6/(3-x)};
\end{axis}
\end{tikzpicture}
    \caption{(Exponent estimates for $d=3$.) Our new upper bounds (blue) vs.\ the conjectured true values (red) and the upper bounds proven in \cite{hutchcroft2020power} (orange) of $2-\eta$ and $\delta$ when $d=3$. The part of the graph where our upper bound coincides exactly with the conjectured true value is represented in purple.
Here we use the numerical values $\alpha_c(3)=2-\eta_\mathrm{SR}(3)\approx 2.0457$ and $\delta_\mathrm{SR}(3)\approx 5.2886$ obtained by applying the scaling and hyperscaling relations to the numerical estimates on the exponents $\nu$ and $\beta/\nu$ obtained by Wang et al.\ in \cite{wang2013bond}. When $\alpha=2.0457\approx \alpha_c(3)$ our upper bound on $\delta$ is about $6.29$ and exceeds the numerical true value by about $20\%$.}
\label{fig:3d}
\end{figure}
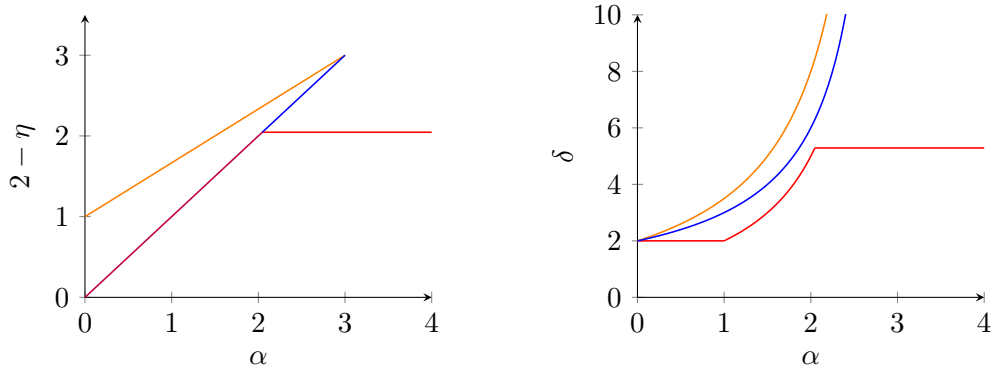

We think of this result as stating that the critical two-point function for long-range percolation on the Euclidean lattice is always dominated on average by the critical two-point function for long-range percolation on the hierarchical lattice. For $d=1$ this inequality is expected to be sharp for all $0<\alpha<1$.\footnote{See \cref{remark:BaumlerBerger}.} Similar bounds have been established (under perturbative criteria) in the high-dimensional case  using the lace expansion \cite{MR3306002,MR4032873,MR2430773}; our results are most interesting when $d\leq 6$ and $\alpha \geq d/3$ so that mean-field critical behaviour is not expected to hold and high-dimensional techniques such as the lace expansion  should not apply. Even in the high-dimensional case, it is notable that we obtain sharp (when $\alpha<d/3$) upper bounds on the two-point function under non-perturbative assumptions, in contrast to lace-expansion based methods.

\begin{remark}
A sharp analysis of the \emph{subcritical} two-point function holding for a very general class of long-range models is given in \cite{MR4248721}.
\end{remark}

\textbf{The tail of the volume.} Applying the methods of \cite{hutchcroft2020power,1808.08940} together with the sharpened control of the two-point function given by \cref{thm:main} yields the following improved power-law bound on the tail of the volume. While this bound is not believed to be sharp, it significantly improves the bound of \cite{hutchcroft2020power} and in fact is rather close to the conjectured true value for $d$ small and $\alpha$ close to $\alpha_c$ as can be seen in \cref{fig:1d,fig:2d,fig:3d}. 

\begin{corollary}
\label{cor:volume}
Let $d\geq 1$, let $J:\Z^d\times \Z^d \to [0,\infty)$ be a symmetric, integrable, translation-invariant kernel, and suppose that there exist constants $0<\alpha<d$ and $c>0$ such that $J(x,y) \geq c\|x-y\|^{-d-\alpha}$ for all distinct $x,y \in \Z^d$. Then there exists a constant $A=A(d,\alpha)$ such that
\[
\bP_{\beta_c}(|K|\geq n) \leq \frac{A}{(c\beta_c)^{d/(2d-\alpha)}} n^{-(d-\alpha)/2d}
\] 
for every $n\geq 1$. In particular, the exponent $\delta$ satisfies $\delta \leq 2d/(d-\alpha)$ if it is well-defined.
\end{corollary}

\textbf{About the proof.}
As noted in \cite[Remark 2.11]{hutchcrofthierarchical}, the same proof used to study the hierarchical lattice in that paper applies directly to long-range percolation on $\Z^d$ to establish that
\begin{equation}
\label{eq:weaker}
\frac{1}{|\Lambda_r|}\sum_{x\in  \Lambda_r} \bP_{\beta_c}(0\leftrightarrow x \text{ inside $\Lambda_r$}) \leq A r^{-d+\alpha} 
\end{equation}
for some constant $A$ and every $r\geq 1$. However, the methods used in \cite[Section 2.3]{hutchcrofthierarchical} to pass from connectivity estimates inside a box to global connectivity estimates break down completely in the Euclidean case, and we are not aware of any techniques to pass from estimates of the form \eqref{eq:weaker} to full-space estimates. Indeed, it is a well-known (and not too difficult) folklore theorem that $\frac{1}{|\Lambda_r|}\sum_{x\in  \Lambda_r} \bP_{p_c}(0\leftrightarrow x \text{ inside $\Lambda_r$}) \to 0$ as $r\to\infty$ for critical nearest-neighbour bond percolation in every dimension $d\geq 2$, whereas proving the corresponding full-space estimate would imply the continuity of the phase transition and remains very much open.
To circumvent this problem, we will instead set up the renormalization argument used to prove \eqref{eq:weaker} in a more subtle and technical way, leading to a stronger (and more technical) estimate that can, with work, be used to deduce \cref{thm:main} by an elaboration of the argument used in the hierarchical case.


A key idea powering the proof will be to write the Euclidean kernel $J$ as a sum of a hierarchical kernel $H_\sigma$ and a remainder term $R_\sigma$, both depending on an $L$-adic hierarchical decomposition $\sigma$ of $\Z^d$ that we are free to choose. We will  then consider connectivity using not only those edges lying inside a given box but also many edges outside the box. Indeed, we take as many of these edges as we can without breaking our renormalization argument, which analyzes the effect of adding long edges at each successive $L$-adic scale. Taking more edges in the restricted estimate in this way makes it much easier to pass from a restricted estimate to a full-space estimate. On the other hand, the non-translation-invariant nature of the hierarchical decomposition causes various problems due to the restricted models depending in a complicated way on the entire choice of decomposition, breaking transitivity. This leads to various subtle technical problems that must be circumvented to push the proof through. This is mostly accomplished by careful choice of definitions, leading to many definitions that may seem somewhat unnatural at first but which have been reverse-engineered precisely to make the proof go through cleanly. 


While we have written the paper in a self-contained way, the reader is likely to better appreciate the proof if they are already familiar with that of \cite{hutchcrofthierarchical}.




\section{Proof}

\subsection{The hierarchical decomposition}

\begin{figure}[t]
\centering
\includegraphics[width=0.4\textwidth]{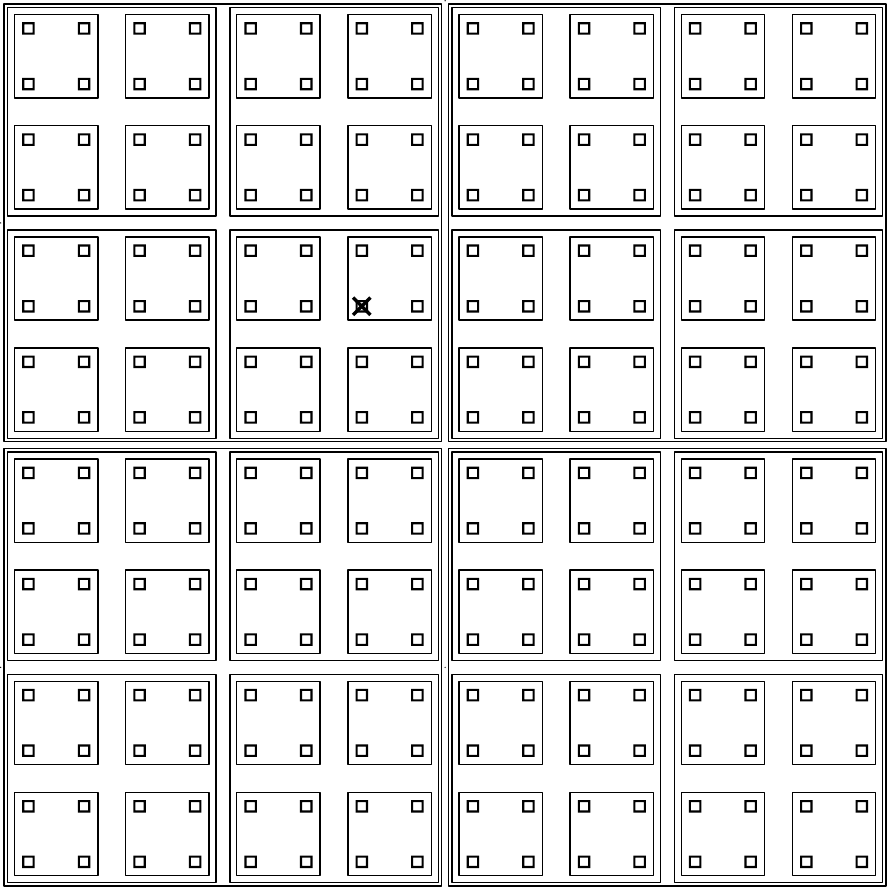}\qquad \qquad 
\includegraphics[width=0.4\textwidth]{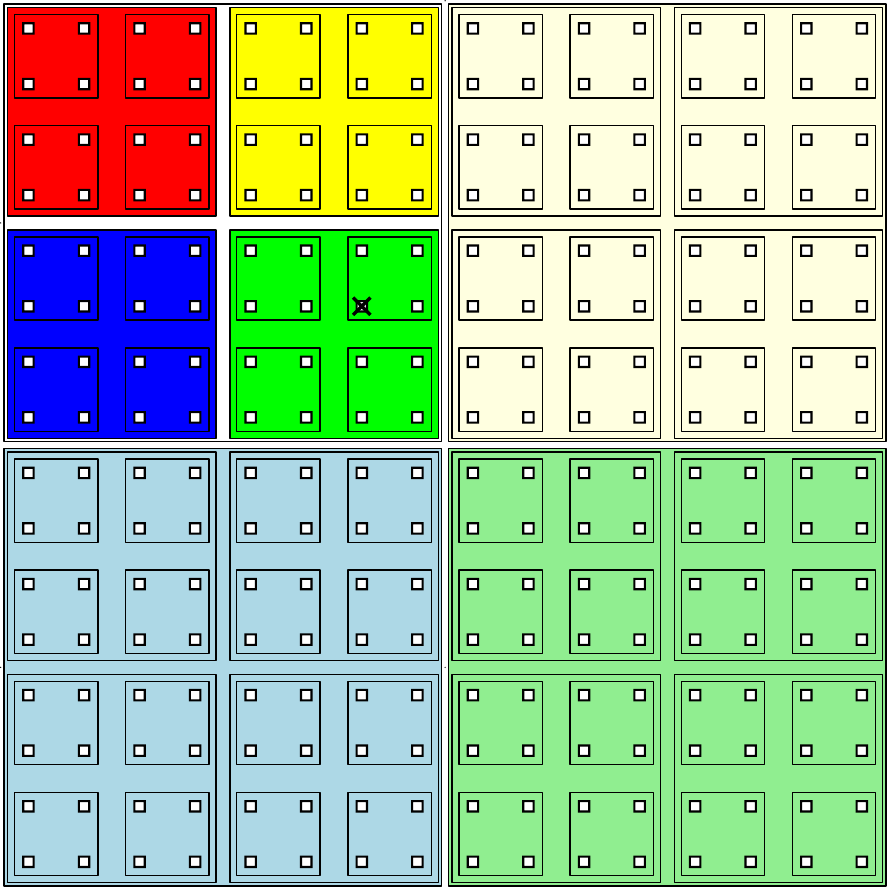}
\caption{Left: The hierarchical decomposition of $\Z^2$ with $L=2$ and $\sigma$ encoded by $((0,0)$,$(1,1)$, $(1,0)$,$(0,1),\ldots)$, where the origin is marked by a cross. Right: The configuration $\eta_{B^\sigma_2}$ is permitted to use only those hierarchical edges whose endpoints are both contained in a box of the same colour.}
\label{fig:decomp}
\end{figure}

\noindent Fix $d\geq 1$ and $L\geq 2$. Each sequence $\sigma = (\sigma_1,\sigma_2,\ldots) \in \Sigma = \Sigma_{d,L}:=(\{0,\ldots,L-1\}^d)^\N$ determines an $L$-adic hierarchical decomposition of $\Z^d$ as follows: For each $n \geq 0$ we define an \textbf{$n$-block} to be a set of the form $\Z^d \cap (\sum_{m=1}^n \sigma_{m,i} L^m+ \prod_{i=1}^d \bigl[k_i L^n,\, (k_i+1) L^n-1])$ for $k_1,\ldots,k_d \in \Z$ and define $\cB_n^\sigma$ to be the partition of $\Z^d$ into $n$-blocks
\[
\cB_n^\sigma = \Biggl\{ \Z^d \cap \Biggl(\sum_{m=1}^n \sigma_{m,i} L^m+\prod_{i=1}^d \Bigl[k_i L^n, (k_i+1) L^n-1  \Bigr]\Biggr) : k=(k_1,\ldots,k_d) \in \Z^d \Biggr\}.
\]
See \cref{fig:decomp} for an illustration.
We reserve the term \emph{block} for $L$-adic boxes of this particular form as determined by $\sigma$, using the term \emph{box} for other sets of the form $[0,r]^d + x$ for $r\geq 0$ and $x\in \Z^d$.
Note that $0$-blocks are simply singleton sets, so that $\cB^\sigma_0=\{\{x\}:x\in \Z^d\}$. We also write $\cB^\sigma=\bigcup_{n\geq 0} \cB^\sigma_n$ for the collection of all blocks, noting again that this collection depends on the choice of $\sigma \in \Sigma$. 
For each $x\in \Z^d$ and $n\geq 0$, we define $B^\sigma_n(x)$ to be the unique $n$-block containing $x$, writing $B_n^\sigma=B^\sigma_n(0)$ for the $n$-block containing the origin. For each $n\geq 1$, every $n$-block decomposes into exactly $L^d$ $(n-1)$-blocks, which we refer to as the \textbf{children} of the block. We refer to the blocks \emph{strictly containing} a given block as the \textbf{ancestors} of that block, and the blocks \emph{strictly contained} in a given block as the \textbf{descendants} of that block. 
Given an $n$-block $B$ for some $n\geq 0$, the unique $(n+1)$-block containing $B$ is called the \textbf{parent} of $B$ and is denoted $\sigma(B)$. We think of $\Sigma$ as the set of $L$-adic hierarchical decompositions of $\Z^d$, and note that for each $\sigma \in \Sigma$ and $x\in \Z^d$ we can consider the translated configuration $\sigma+x\in \Sigma$ which is defined by the property that
\begin{equation}
\label{eq:translation_def}
\cB^{\sigma+x}_n = \{ B + x : B \in \cB^\sigma_n\}  \qquad \text{ for every $n\geq 0$.}
\end{equation}


 Given $\sigma \in \Sigma$ we define the \textbf{hierarchical ultrametric} $d_\sigma$ on $\Z^d$ by
\[
d_\sigma(x,y) = \begin{cases} L^{h_\sigma(x,y)} & h_\sigma(x,y) \geq 1 \\ 
0 & h_\sigma(x,y) = 0 \end{cases} \quad \text{ where } \quad 
\begin{array}{c}\text{$h_\sigma(x,y)$ is minimal such that there exists} \\ 
\text{an $h_\sigma(x,y)$-block containing $x$ and $y$.}
 \end{array}
\]
We stress that all of these definitions depend on the choice of sequence $\sigma$ which encodes the hierarchical $L$-adic decomposition of $\Z^d$. Note that for some (non-generic) choices of $\sigma$, such as the all-zero sequence, there exist vertices with $d_\sigma(x,y)=h_\sigma(x,y)=\infty$; this will not cause us any problems.
It is easily verified that $d_\sigma$ does indeed define an (extended) ultrametric, and in particular satisfies the ultrametric triangle inequality $d_\sigma(x,z) \leq \max\{d_\sigma(x,y),d_\sigma(y,z)\}$ for every $x,y,z\in \Z^d$. (In fact for generic choices of $\sigma$ the metric space $(\Z^d,d_\sigma)$ is isometric to the hierarchical lattice $\bbH^d_L$, but with the collection of maps realising this isometry depending on $\sigma$.) Note also that $d_\sigma(x,y) \geq \|x-y\|$ for every $x,y\in \Z^d$ and $\sigma\in \Sigma$.


Fix a symmetric, integrable, translation-invariant kernel $J:\Z^d\times \Z^d \to[0,\infty)$ and suppose that there exist constants $c>0$ and $0<\alpha<d$ such that $J(x,y)\geq c\|x-y\|^{-d-\alpha}$ for every pair of distinct points $x,y\in \Z^d$. We will consider $d$, $L$, $J$, $\alpha$, and $c$ to be fixed for the remainder of the paper, and will later need to assume that $L$ is larger than some constant $L_0=L_0(d,\alpha)$. For each $\sigma \in \Sigma$ and $x,y\in \Z^d$ we define the symmetric kernels $H_\sigma$ and $R_\sigma$ by
\[
H_\sigma(x,y) = \begin{cases} c d_\sigma(x,y)^{-d-\alpha} & x\neq y\\
0 & x=y \end{cases} \qquad \text{ and } \qquad R_\sigma(x,y) = J(x,y) - H_\sigma(x,y)
\]
so that $J(x,y)=H_\sigma(x,y)+R_\sigma(x,y)$ for every pair of distinct points $x,y\in \Z^d$ and $\sigma \in \Sigma$. The non-negativity of $R_\sigma$ is ensured since $J(x,y)\geq c\|x-y\|^{-d-\alpha} \geq cd_\sigma(x,y)^{-d-\alpha}$ for every pair of distinct points $x,y\in \Z^d$. We think of this as a decomposition of the Euclidean kernel $J$ into a hierarchical term $H_\sigma$ and a remainder term $R_\sigma$, noting that $R_\sigma(x,y)\gg H_\sigma(x,y)$ for $x$ and $y$ on opposite sides of the boundary of a large block.
Given $\sigma\in \Sigma$, $n \geq 1$, and an $n$-block $B \in \cB^\sigma_n$ we also define the kernel
\[
H_{B}(x,y):=  cL^{-(d+\alpha)n} \mathbbm{1}(x,y\in B \text{ and } h_\sigma(x,y)=n)
\]
for each $x,y\in \Z^d$, so that
$H_\sigma (x,y) = \sum_{n\geq 1} \sum_{B\in \cB^\sigma_n}H_{B}(x,y)$ for each $x,y\in \Z^d$. Note that the kernels $H_\sigma$ and $R_\sigma$ depend heavily on the choice of $\sigma \in \Sigma$ and are \emph{not} translation-invariant for fixed $\sigma$. They are, however, \emph{translation-covariant} in the sense that
\begin{equation}
H_{\sigma+x} (a+x,b+x) = H_\sigma(a,b) \quad \text{ and } \quad R_{\sigma+x} (a+x,b+x) = R_\sigma(a,b)
\end{equation}
for every $x,a,b\in \Z^d$ and $\sigma\in \Sigma$.



Fix $\beta \geq 0$ and $\sigma \in \Sigma$. Let $\omega_R$ be long-range Bernoulli percolation on $\Z^d$ with kernel $R_\sigma$ and parameter $\beta$ and for each non-singleton block $B$ let $\omega_B$ be a long-range Bernoulli percolation configuration on $\Z^d$ with kernel $H_{B}$ and parameter $\beta$, where the $\omega_B$ are all independent of each other and of $\omega_R$. The union $\omega$ of the configuration $\omega_R$ with all of the configurations $\omega_B$ is equal in distribution to long-range Bernoulli percolation on $\Z^d$ with kernel $J$, and we write $\bP_{\beta,\sigma}$ for the joint law of $\omega_R$ and $((\omega_B)_{B\in \cB^\sigma_{n}})_{n\geq 1}$.


As discussed above, we will want to define something that plays the role of `the cluster inside a block' but where we will also want to include as many edges \emph{outside} the block as possible without breaking the proof strategy of \cite{hutchcrofthierarchical}. These considerations lead to the following definition: For each block $B\in \cB^\sigma$ we define
\[
\eta_B := \omega_R \; \cup 
\left(\bigcup_{m=1}^\infty \bigcup \left\{ \omega_{B'} : B'\in \cB^\sigma_m \text{ is not an ancestor of $B$} \right\} \right),
\]
where $\bigcup \{A_i:i\in I\}:=\bigcup_{i\in I}A_i$, 
which is distributed as long-range Bernoulli percolation with kernel
\[
J_{\sigma,B}(x,y):=R_\sigma(x,y) + \sum_{m=1}^\infty \sum_{B'\in \cB^\sigma_m} H_{B'}(x,y)\mathbbm{1}(B' \text{ not an ancestor of $B$}).
\]
In other words, $\eta_B$ includes all edges associated to the remainder kernel and all edges associated to the part of the hierarchical kernel corresponding to blocks that are either contained in $B$ or are disjoint from $B$ (i.e., those hierarchical edges whose endpoints have the same colour in \cref{fig:decomp}). In particular, $\eta_B=\eta_{B'}$ when $B$ and $B'$ are siblings (i.e., have the same parent) and
\begin{equation}\eta_{\sigma(B)} =\eta_B \cup \omega_{\sigma(B)}
\label{eq:eta_recursion}\end{equation}
for every block $B$, where we recall that $\sigma(B)$ denotes the parent of $B$. This fact will be the basis of all our renormalization arguments. Note also that $J_{\sigma,B}$ transforms covariantly under translation in the sense that
\begin{equation}
\label{eq:translation_covariance}
J_{\sigma+x,B+x}(a+x,b+x) = J_{\sigma,B}(a,b)
\end{equation}
for every $a,b,x\in \Z^d$, $\sigma \in \Sigma$ and $B \in \cB^\sigma$, so that the law of $\eta_B$ enjoys a similar translation-covariance property.

\subsection{The maximum cluster size}

Given $\sigma \in \Sigma$ and a block $B\in \cB_n^\sigma$ for some $n\geq 0$, we write 
\[
|K^{\max}_B| = \max\bigl\{|K\cap B| : K \text{ is a cluster of $\eta_B$}\bigr\}.
\]
(This is a slight abuse of notation since $K^{\max}_B$ is not well-defined \emph{as a set} in the case of a tie. This will not cause any problems.) We stress that although we consider the intersections of clusters of $\eta_B$ with $B$, these clusters need not be contained in $B$, and may contain both arbitrary edges from $\omega_R$ and edges from $\omega_{B'}$ for any block $B'$ that is not an ancestor of $B$ (including e.g.\ small blocks very far away from $B$).
We  define the \emph{typical value} of $|K^{\max}_B|$ to be
\[
M_B=M_{B,\beta,\sigma}:= \min\Bigl\{m \geq 1 : \bP_{\beta,\sigma}\left(|K^\mathrm{max}_B| \geq m\right) \leq \frac{1}{e}\Bigr\},\]
noting that we always have $M_B \geq 2$.
As in \cite{hutchcrofthierarchical}, our analysis will rely crucially on the universal tightness theorem of \cite[Theorem 2.2]{hutchcroft2020power}, which implies that $|K^{\max}_B|$ is always of the same order as its typical value $M_B$ with high probability. This theorem, which applies to percolation on arbitrary weighted graphs and hence to long-range percolation on $\Z^d$ with an arbitrary symmetric kernel,  yields in our context that the inequalities
\begin{equation}
\bP_{\beta,\sigma}\Bigl(|K_B^\mathrm{max}| \geq \lambda M_B\Bigr) \leq \exp\left(-\frac{1}{9}\lambda \right)
\label{eq:BigClusterUnrooted}
\qquad \text{and} \qquad \bP_{\beta,\sigma}\Bigl(|K_B^\mathrm{max}| < \eps M_B \Bigr) \leq 27 \eps 
\end{equation}
hold for every $\sigma \in \Sigma$, $n \geq 0$, $B\in \cB_n^\sigma$, $\lambda \geq 1$ and $0<\eps \leq 1$. 
Moreover, if we write $K_B(x)$ for the cluster of $x$ in $\eta_{B}$ for each $x\in \Z^d$ then we also have that
\begin{equation}
\label{eq:BigClusterRooted}
\bP_{\beta,\sigma}\Bigl(|K_B(x) \cap B| \geq \lambda M_B\Bigr) \leq  \bP_{\beta,\sigma}\Bigl(|K_B(x) \cap B| \geq  M_B\Bigr) \exp\left(1-\frac{1}{9}\lambda \right)
\end{equation}
for every $\lambda \geq 1$. These inequalities make upper bounds on $M_B$ (which is a kind of median) very useful for the establishment of upper bounds on \emph{moments} of related quantities; this plays a very important technical role in several of our proofs. It follows in particular from \eqref{eq:BigClusterUnrooted} that the mean of $|K_B^{\max}|$ is always of the same order as its typical value in the sense that
\begin{equation}
\frac{M_B}{2e} \leq \frac{M_B-1}{e}\leq \bE_{\beta,\sigma}|K_B^{\max}| \leq  \left(1+\int_1^\infty e^{-\lambda/9}\dif \lambda \right) M_B \leq 10 M_B
\label{eq:mean_and_median}
\end{equation}
for every $\sigma\in \Sigma$, every $\beta\geq 0$, and every block $B \in \cB^\sigma$.


\medskip

As in \cite{hutchcrofthierarchical}, our first goal will be to establish an upper bound on $M_B$ for $\beta<\beta_c$ using what we call a \emph{runaway observable argument}. That is, we will show that if $M_B$ is much larger than we believe it should be for some block $B$ then $M_{\sigma(B)}$ is larger than it should be by an even larger factor, so that, inductively, the quantities associated to the ancestors of $B$ blow up rapidly as we pass through the generations. This rapid growth will contradict the \emph{sharpness of the phase transition}, which states in particular that the expected size of the cluster of the origin is finite for $\beta < \beta_c$ \cite{aizenman1987sharpness,duminil2015new,1901.10363}, so that in fact the anticipated bound on $M_B$ can never be exceeded.

\medskip

Two problems arise immediately when adapting the arguments of \cite{hutchcrofthierarchical} to our new setting: First, the non-transitivity of the hierarchical decomposition means that different blocks of the same size may have different values of $M_B$. Second, since $|K_B^{\max}|$ may depend on edges that are not contained in $B$, the random variables $|K_{B_1}^{\max}|$ and $|K_{B_2}^{\max}|$ need not be independent for two disjoint blocks of the same size. We will see that the second issue can be circumvented fairly easily by an additional application of the universal tightness theorem, while the first is more serious.

\medskip

To deal with the problem of non-transitivity, we will bound $M_B$ not for arbitrary blocks, but only for those blocks that are \emph{ancestrally good}, a notion we now define. Given $\sigma \in \Sigma$, we say that two blocks are \textbf{siblings} if they share a parent, and say that a block $B$ is \textbf{good} if
\begin{align}
&\bE_{\beta,\sigma}|K_{B}^{\max}| \leq  \bE_{\beta,\sigma}|K_{B'}^{\max}|  &\hspace{-1.5cm}\text{ for at least $\left\lfloor \frac{1}{2} L^d\right\rfloor-1$ siblings $B'$ of $B$ and}\label{eq:good_def_1}\\
&\sum_{x,y\in B} \bP_{\beta,\sigma}(x \leftrightarrow y \text{ in $\eta_B$}) \leq  \sum_{x,y\in B'} \bP_{\beta,\sigma}(x \leftrightarrow y \text{ in $\eta_{B'}$})  &\text{ for at least one sibling $B'$ of $B$}.\label{eq:good_def_2}
\end{align}
Note that the two sets of siblings required by these two conditions need not be the same.

\begin{lemma}
\label{lem:good_children}
Every non-singleton block has at least $\frac{1}{2}L^d$ children that are good, and in particular has at least one such child.
\end{lemma}

\begin{proof}[Proof of \cref{lem:good_children}]
Every non-singleton block has $L^d$ children. Of these children, at most one does not satisfy \eqref{eq:good_def_2}, and at most $\lfloor \frac{1}{2} L^d \rfloor-1$ do not satisfy \eqref{eq:good_def_1}. As such, the total number of children that are \emph{not} good is at most $\lfloor \frac{1}{2} L^d \rfloor$. This is equivalent to the claim.
\end{proof}


We say that a block is \textbf{ancestrally good} if it is good and all of its ancestors are good (in which case all of its ancestors are ancestrally good). Note that (ancestral) goodness of a block may depend both on $\sigma$ and the parameter $\beta$.

\medskip

The first basic fact we will need is that we can choose $\sigma$ so that ancestrally good blocks exist. Note that the choice of $\sigma$ may depend on $\beta$.

\begin{prop}
\label{lem:ancestrally_minimal}
For each $0\leq \beta<\beta_c$ there exists $\sigma \in \Sigma$ such that the singleton block $\{0\}$ is ancestrally good.
\end{prop}

Before proving this proposition, we first prove the following auxiliary continuity lemma.

\begin{lemma}
\label{lem:weak_continuity}
For each $0 \leq \beta < \beta_c$ and $n\geq 0 $, the 
set of $\sigma\in \Sigma$ such that $B^\sigma_n$ is good is closed in the product topology.
\end{lemma}

\begin{proof}[Proof of \cref{lem:weak_continuity}]
Fix $0 \leq \beta < \beta_c$ and $n\geq 0 $. First note that
if that $\sigma,\tau\in \Sigma$ and $N \geq n \geq 1$ are such that $\sigma_i=\tau_i$ for every $i\leq N$ then the sets of $m$-blocks $\cB^\sigma_m$ and $\cB^\tau_m$ coincide for every $0\leq m \leq N$, so that $B^\sigma_n=B^\tau_n$. 
 For each set $B$ that is an $n$-block of some $\sigma \in \Sigma$, let $\Sigma_B$ be the set of $\sigma$ for which $B\in \cB^\sigma_n$, so that any two elements of $\Sigma_B$ agree in their first $n$ coordinates. Since the block $B^\sigma_n$ and its set of siblings depend continuously on $\sigma$ and the conditions defining a good block are closed, it suffices to prove that for each such block $B$ and point $x\in B$ the expectations $\bE_{\beta,\sigma}|K_{B}^{\max}|$ and $\sum_{y\in B}sa\bP_{\beta,\sigma}(x \leftrightarrow y \text{ in $\eta_B$})=\bE_{\beta,\sigma}|K_{B}(x) \cap B|$ depend continuously on $\sigma \in \Sigma_B$. By bounded convergence, it suffices to prove that the distribution of $K_B(x)$ depends continuously on $\sigma \in \Sigma_B$ for every $x\in \Z^d$.

Fix such a $B$ and suppose that $\sigma,\tau\in \Sigma_B$ and $N \geq n \geq 1$ are such that $\sigma_i=\tau_i$ for every $i\leq N$.
We may construct a coupled pair of random variables $(\eta_\sigma,\eta_\tau)$ such that $\eta_\sigma$ has the law of $\eta_B$ under $\bP_{\beta,\sigma}$, $\eta_\tau$ has the law of $\eta_B$ under $\bP_{\beta,\tau}$, the random variables $(\eta_\sigma(\{x,y\}),\eta_\tau(\{x,y\}))$ and $(\eta_\sigma(\{a,b\}),\eta_\tau(\{a,b\}))$ are independent when $\{x,y\}\neq \{a,b\}$, and 
\[
\P\Bigl(\eta_\sigma(\{x,y\}) \neq \eta_\tau(\{x,y\})\Bigr)=1-\exp(-\beta|J_{\sigma,B}(x,y)-J_{\tau,B}(x,y)|) \leq \beta|J_{\sigma,B}(x,y)-J_{\tau,B}(x,y)|
\]
for every $x,y\in \Z^d$, where we write $\P$ for the joint law of this coupled pair of random variables.
We can also estimate
\begin{align}
|J_{\sigma,B}(x,y)-J_{\tau,B}(x,y)| &\leq |R_{\sigma}(x,y)-R_\tau(x,y)| + \sum_{m=N+1}^\infty \sum_{B'\in \cB^\sigma_m} H_{B'}(x,y) + \sum_{m=N+1}^\infty \sum_{B'\in \cB^\tau_m} H_{B'}(x,y) \nonumber\\
&\leq 3c \min \{L^{-(d+\alpha)(N+1)},\|x-y\|^{-d-\alpha}\}
\end{align}
for every distinct $x,y\in \Z^d$, where in the second line we used that
\begin{multline}
|R_{\sigma}(x,y)-R_\tau(x,y)| = |H_\sigma(x,y)-H_\tau(x,y)|=c|d_\sigma(x,y)^{-d-\alpha}-d_\tau(x,y)^{-d-\alpha}|\\ \leq c \min \{L^{-(d+\alpha)(N+1)},\|x-y\|^{-d-\alpha}\}.
\end{multline}
It follows by an elementary calculation that there exists a constant $A=A(d,\alpha,c)$ such that
\begin{equation}
\label{eq:defects}
\sum_{y\in \Z^d} |J_{\sigma,B}(x,y)-J_{\tau,B}(x,y)|  \leq A L^{-\alpha N}
\end{equation}
for every $x \in \Z^d$. The only important feature of this bound is that it tends to zero as $N\to\infty$ uniformly in $x$.


Fix an enumeration $\{x_1,x_2,\ldots\}$ of $\Z^d$ and let $\opleq$ be the associated total order on $\Z^d$.
Suppose that we explore the cluster of a vertex $x$ in both $\eta_\sigma$ and $\eta_\tau$ one vertex at a time as follows: We first reveal the status of every edge incident to $x$. At each subsequent step we choose the $\opleq$-minimal point of $\Z^d$ that has not already been chosen and that is incident to a revealed open edge and reveal the status of every edge incident to that vertex.  We stop when no such vertices remain. If we run this process for $k$ steps, the probability we find that the cluster of $x$ differs in $\eta_\sigma$ and $\eta_\tau$ is at most $A \beta L^{-\alpha N} k$ by \eqref{eq:defects} and Markov's inequality, so that
\begin{equation}
\label{eq:connectivity_continuity}
\P(\text{the clusters of $x$ in $\eta_\sigma$ and $\eta_\tau$ are distinct})\leq \P_{\beta}(|K_x|\geq k) + A \beta L^{-\alpha N} k
\end{equation}
for every $x \in \Z^d$ and $k\geq 1$. Since $\beta <\beta_c$ we have that $\P_{\beta}(|K_x|\geq k) \to 0$ as $k\to \infty$. Thus, taking, say, $k=N$ we deduce that the left hand side of \eqref{eq:connectivity_continuity} is small when $N$ is large, uniformly in $\sigma$, $\tau$, and $x$. This establishes the desired distributional continuity of $K_B(x)$ and concludes the proof. \qedhere

\end{proof}


\begin{proof}[Proof of \cref{lem:ancestrally_minimal}]
Fix $0 \leq \beta < \beta_c$ and an arbitrary sequence $\sigma_0 \in \Sigma$. 
For each $n\geq 1$, every $n$-block must have at least one good child by \cref{lem:good_children}. By picking good children of good children recursively, we may find for each $n\geq 1$ a vertex $x_n$ such that $B_m^{\sigma_0}(x_n)$ is good for every $0\leq m \leq n$. Since $J$ is translation-invariant, we deduce that for each $n\geq 0$ there exists $\sigma_n=\sigma-x_n \in \Sigma$ such that $B^{\sigma_n}_m=B^{\sigma_n}_m(0)$ is good for every $0\leq m \leq n$. By compactness of $\Sigma$, the sequence $(\sigma_n)_{n\geq 0}$ has a subsequence converging pointwise to some limit $\sigma_\infty \in \Sigma$, and it follows from \cref{lem:weak_continuity} that $\{0\}$ is ancestrally good under $\sigma_\infty$ as claimed. \qedhere

\end{proof}


Now that we know that $\sigma$ can be chosen so that $\{0\}$ is ancestrally good, our next goal will be to prove an upper bound on $M_B$ for ancestrally good blocks when $L$ is large.

\begin{prop}
\label{prop:maximum_upper}
There exists an integer $L_0=L_0(d,\alpha) \geq 2$ such that the following holds: If $L \geq L_0$ then
there exists a constant $A=A(d,L,\alpha)$ such that 
\[
M_B^2 < \frac{A}{c\beta} L^{(d+\alpha)n}
\]
for every $0< \beta < \beta_c$, $\sigma\in \Sigma$, $n\geq 0$, and every ancestrally good $n$-block $B\in \cB^\sigma_n$.
\end{prop}

Since $M_B \geq 2$ for every block $B$, taking $L=L_{0}$ and applying \cref{prop:maximum_upper} in the case $n=0$ yields the following immediate corollary in conjunction with \cref{lem:ancestrally_minimal}. 

\begin{corollary}
\label{cor:betabound} 
Let $J:\Z^d\times \Z^d \to [0,\infty)$ be a symmetric, integrable function, let $0<\alpha<d$, and suppose that there exists a positive constant $c$ such that $J(x,y) \geq c\|x-y\|^{-d-\alpha}$ for every distinct $x,y\in \Z^d$. Then there exists a constant $\beta_1=\beta_1(d,\alpha)$ such that $c\beta_c \leq \beta_1$.
\end{corollary}

\begin{proof}[Proof of \cref{cor:betabound}]
Let $L_0=L_0(d,\alpha)\geq 2$ and $A=A(d,L_0,\alpha)$ be as in \cref{prop:maximum_upper} and let $\beta<\beta_c$.
By \cref{lem:ancestrally_minimal} there exists an $L_0$-adic hierarchical partition $\sigma$ of $\Z^d$ such that the $0$-block $\{0\}$ is ancestrally good. Since $M_{\{0\}}=2$ by definition, it follows from \cref{prop:maximum_upper} that
$c\beta < A/4$, and since $\beta<\beta_c$ was arbitrary it follows that $c\beta_c \leq A/4$. The claim follows since $L_0$ and hence $A$ were chosen to depend only on $d$ and $\alpha$.
\end{proof}

We will deduce \cref{prop:maximum_upper} from the sharpness of the phase transition together with the following renormalization lemma. Recall that $\sigma(B)$ denotes the parent of the block $B$ in the hierarchical decomposition $\sigma$.

\begin{lemma}[Renormalization of the maximum cluster size]
\label{lem:maximum_renormalization}
There exists an integer $L_0=L_0(d,\alpha) \geq 2$ such that the following holds: 
If $L \geq L_0$ then there exists a constant $A=A(d,L,\alpha)$ such that the implication
\[\left(\text{$B$ good and }M_B^2 \geq \frac{A}{c\beta} L^{(d+\alpha)n}\right) \Rightarrow \left(M_{\sigma(B)}^2 \geq \frac{A}{c\beta} L^{(d+\alpha)(n+1)} \right)\]
holds for every $\beta > 0$, every $\sigma \in \Sigma$, every $n\geq 0$, and every $n$-block $B \in \cB^\sigma_n$.
\end{lemma}

\begin{proof}[Proof of \cref{lem:maximum_renormalization}]
Fix $\eps$, $\beta$, and $\sigma$ and let $B$ be a good $n$-block for some $n\geq 0$. Since $B$ is good there are at least $\lfloor L^d/2 \rfloor-1$ siblings $B'$ of $B$ with $\bE_{\beta,\sigma}|K_{B'}^{\max}| \geq \bE_{\beta,\sigma}|K_{B}^{\max}|$. Let $\ell$ be the number of these siblings, which we enumerate $B_1,\ldots,B_\ell$ and write $B_0=B$. For each $0\leq i \leq \ell$ we have by \eqref{eq:mean_and_median} that
\[
M_{B_i} \geq \frac{1}{10}\bE_{\beta,\sigma}|K_{B_i}^{\max}| \geq  \frac{1}{10}\bE_{\beta,\sigma}|K_{B}^{\max}|  \geq \frac{1}{20 e} M_{B} \geq \frac{1}{60} M_B
\]
and hence by \eqref{eq:BigClusterUnrooted} that
\[
\bP_{\beta,\sigma}\left(|K_{B_i}^{\max}| < 2^{-14} M_{B}   \right) \leq \bP_{\beta,\sigma}\left(|K_{B_i}^{\max}| < \frac{1}{27 \cdot 8} M_{B_i}   \right) \leq \frac{1}{8}
\]
for every $0\leq i \leq L^d-1$, where we used that $27\cdot8\cdot 60 \leq 2^{14}$ in the first inequality.
It follows by Markov's inequality that
\begin{equation}
\label{eq:most_largest_clusters_large}
\bP_{\beta,\sigma}\left(\#\left\{ 0 \leq i \leq \ell : |K_{B_i}^{\max}| < 2^{-14} M_{B_0}\right\} \leq \frac{\ell+1}{2}   \right) \geq 1-\frac{1}{4}.
\end{equation}
Let $\sA$ be the event whose probability is estimated on the left hand side of \eqref{eq:most_largest_clusters_large} and let $\cF$ be the sigma-algebra generated by $\eta_B$, which we recall is equal to $\eta_{B_i}$ for every $1\leq i \leq \ell$.
Observe that $\sA$ is measurable with respect to $\cF$ and that $\cF$ is independent of the configuration $\omega_{\sigma(B)}$ since this configuration does not contribute to any of the configurations $\eta_{B_i}$ for $0\leq i \leq \ell$. 
For each $0 \leq i \leq \ell$, let $D_i \subseteq B_i$ be such that $|D_i|=|K_{B_i}^{\max}|$ and $D_i = K \cap B_i$ for some cluster $K$ of $\eta_{B_i}$, where we break ties in an arbitrary $\cF$-measurable way (e.g. using an enumeration of $\Z^d$ that is chosen in advance). For each $0\leq i,j \leq \ell$, the conditional probability given $\cF$ that $D_i$ is connected to $D_j$ by an edge of $\omega_{\sigma(B)}$ is equal to
\[1-\exp\left[-c \beta L^{-(d+\alpha)(n+1)} |D_i| |D_j| \right]. \]
Thus, it follows by a union bound that the conditional probability given $\cF$ that $D_i$ is connected to $D_j$ by an edge of $\omega_{\sigma(B)}$ for every $i$ and $j$ with $|D_i|,|D_j| \geq 2^{-14} M_B$ is a least
\[1-\binom{\ell+1}{2} \exp\left[-2^{-28}c \beta  L^{-(d+\alpha)(n+1)} M_B^2 \right], \]
so that
\[
\bP_{\beta,\sigma}\left(|K_{\sigma(B)}^{\max}| \geq 2^{-15} (\ell+1) M_B \mid \sA\right) \geq 1-\binom{\ell+1}{2}\exp\left[-2^{-28}c \beta L^{-(d+\alpha)(n+1)} M_B^2 \right].
\]
Since $\ell +1 \geq \lfloor L^d/2\rfloor$ and $\alpha <d$, there exists an integer $L_0=L_0(d,\alpha) \geq 2$ such that if $L\geq L_0$ then  
 $2^{-15} (\ell+1) \geq L^{(d+\alpha)/2}$, so that if $L\geq L_0$ then
\[
\bP_{\beta,\sigma}\left(|K_{\sigma(B)}^{\max}| \geq  L^{(d+\alpha)/2} M_B \mid \sA\right) \geq 1-L^{2d}\exp\left[-2^{-28}c \beta L^{-(d+\alpha)(n+1)} M_B^2 \right].
\]
It follows that if $L\geq L_0$ and $A$ is such that $M_B^2 \geq \frac{A}{c\beta} L^{(d+\alpha)n}$ then
\begin{equation}
\bP_{\beta,\sigma}\left(|K_{\sigma(B)}^{\max}| \geq  L^{(d+\alpha)/2} M_B \right) \geq 1-L^{2d}\exp\left[-\frac{A}{2^{28}L^{d+\alpha}} \right] - \frac{1}{4}.
\label{eq:max_renormalization_final}
\end{equation}
If $A$ is chosen sufficiently large as a function of $d$, $L$, and $\alpha$ then the right hand side of \eqref{eq:max_renormalization_final} is larger than $1/e$, so that that $M_{\sigma(B)} \geq L^{(d+\alpha)/2} M_B$ whenever $L\geq L_0$ and $M_B^2 \geq \frac{A}{c\beta} L^{(d+\alpha)n}$. This completes the proof.
\end{proof}

It remains to deduce \cref{prop:maximum_upper} from \cref{lem:maximum_renormalization}.

\begin{proof}[Proof of \cref{prop:maximum_upper}]
Let $L_0$ and $A$ be as in \cref{lem:maximum_renormalization} and suppose that $L\geq L_0$. Fix  $\sigma \in \Sigma$ and $0\leq \beta <\beta_c$ and let $B$ be an ancestrally good $n$-block for some $n\geq 0$. (We may assume that $\sigma$ is such that ancestrally good $n$-blocks exist, the claim being vacuous otherwise.) Write $B=B_n$ and for each $m\geq n$ let $B_m$ be the unique $m$-block that is an ancestor of $B$.
  Suppose for contradiction that $M_B^2 \geq \frac{A}{c\beta} L^{(d+\alpha)/2}$. Applying \cref{lem:maximum_renormalization} recursively yields that
\[
M_{B_m}^2 \geq \frac{A}{c\beta} L^{(d+\alpha)m} \qquad \text{ and hence that } \qquad \bE_{\beta,\sigma} |K_{B_m}^{\max}| \geq \frac{1}{2e} M_{B_m} \geq \frac{1}{2e} \sqrt{\frac{A}{c\beta} L^{(d+\alpha)m}}
\]
for every $m\geq n$. Bounding the volume of the entire cluster of a vertex $x$ in the whole percolation configuration by the volume of the intersection of its cluster in $\eta_{B_m}$ with $B_m$, it follows by transitivity and Jensen's inequality that
\[
 \bE_\beta |K(0)| = L^{-dm}\bE_{\beta}\left[\sum_{x\in B_m}  |K(x)| \right] \geq L^{-dm}\bE_{\beta,\sigma} \left[|K_{B_m}^{\max}|^2\right] \geq \frac{A}{4e^2 c\beta} L^{\alpha m}
\]
for every $m\geq n$. Taking $m\to\infty$ yields that  $\bE_\beta |K(0)|=\infty$, contradicting the sharpness of the phase transition \cite{duminil2015new,1901.10363,aizenman1987sharpness} since $\beta<\beta_c$.
\end{proof}

\subsection{Upper bounds on the restricted two-point function}
\label{sec:restricted_upper}

We next prove an upper bound on the average connection probability between two points in an ancestrally good box under the restricted configuration $\eta_B$. This estimate will be deduced from \cref{prop:maximum_upper} and the universal tightness theorem via a further runaway observable argument. 

\begin{prop}
\label{prop:susceptibility_upper}
There exists an integer $L_0=L_0(d,\alpha) \geq 2$ such that the following holds: If $L \geq L_0$ then
there exists a constant $A=A(d,L,\alpha)$ such that 
\[
\sum_{x,y\in B} \bP_{\beta,\sigma}(x\leftrightarrow y \text{ in $\eta_B$}) < \frac{A}{c\beta} L^{(d+\alpha)n}
\]
for every $0< \beta < \beta_c$, $\sigma\in \Sigma$, $n\geq 0$, and every ancestrally good $n$-block $B\in \cB^\sigma_n$.
\end{prop}

Note that the quantity estimated here can be written equivalently as
\begin{equation}
\sum_{x,y\in B} \bP_{\beta,\sigma}(x\leftrightarrow y \text{ in $\eta_B$}) = \sum_{x\in B} \bE_{\beta,\sigma}|K_B(x)\cap B|.
\end{equation}
We will deduce \cref{prop:susceptibility_upper} from the following renormalization lemma.


\begin{lemma}
\label{lem:susceptibility_renormalization}
There exists an integer $L_0=L_0(d,\alpha) \geq 2$ such that the following holds: If $L \geq L_0$ then
there exists a positive constant $a=a(d,L,\alpha)$ such that 
\[\sum_{x,y\in \sigma(B)} \bP_{\beta,\sigma}\bigl(x\leftrightarrow y \text{ in $\eta_{\sigma(B)}$}\bigr) \geq a c \beta L^{-(d+\alpha) n} \left(\sum_{x,y\in B} \bP_{\beta,\sigma}(x\leftrightarrow y \text{ in $\eta_B$})\right)^2\]
for every $0< \beta < \beta_c$, $\sigma\in \Sigma$, $n\geq 0$, and every ancestrally good $n$-block $B\in \cB^\sigma_n$.
\end{lemma}

Before proving this lemma we first prove the following analogue of \cite[Lemma 2.9]{hutchcrofthierarchical}. Estimates of this form follow very generally from the universal tightness theorem. We write $a \wedge b = \min\{a,b\}$.

\begin{lemma}[Truncating at the typical maximum]
\label{lem:truncated_susceptibility}
There exists a universal positive constant $a$ such that
\[
\bE_{\beta,\sigma} \left[ |K_B(x) \cap B| \wedge (\lambda M_B) \right] \geq a \lambda  \bE_{\beta,\sigma} |K_B(x) \cap B|\]
for every $\beta \geq 0$, $\sigma\in \Sigma$, $0<\lambda \leq 1$, every block $B$ and every $x\in \Z^d$.
\end{lemma}

\begin{proof}[Proof of \cref{lem:truncated_susceptibility}]
Fix $\beta$, $B$, and $x$ and write $K=K_B(x)\cap B$ and $M=M_B$. For each integer $N$ we have that
\[
\bE_{\beta,\sigma} |K| = \sum_{k = 1}^\infty \P_{\beta,\sigma}(|K| \geq k) \qquad \text{ and } \qquad \bE_{\beta,\sigma} \left[|K| \wedge N\right] = \sum_{k = 1}^N \bP_{\beta,\sigma}(|K| \geq k),
\]
and hence by the universal tightness theorem as stated in \eqref{eq:BigClusterRooted} that if $N \geq M$ then
\begin{align*}
\bE_{\beta,\sigma} |K| - \bE_{\beta,\sigma} \left[|K| \wedge N\right] = \sum_{k=N+1}^\infty \bP_{\beta,\sigma}(|K| \geq k)
&\leq e \bP_{\beta,\sigma}(|K| \geq M) \sum_{k= N+1}^\infty e^{-k/(9M)}.
\end{align*}
It follows by Markov's inequality that if $N \geq 99 M$ then
\begin{align*}
\bE_{\beta,\sigma} |K| - \bE_{\beta,\sigma} \left[|K| \wedge N\right] &\leq \frac{e^{1-11}}{1-e^{-1/(9M)}} \bP_{\beta,\sigma}(|K| \geq M) \\&\leq \frac{e^{-10}}{(1-e^{-1/(9M)}) M} \bE_{\beta,\sigma}|K| \leq \frac{1}{2} \bE_{\beta,\sigma} |K|,
\end{align*}
where the final inequality follows by calculus since $M \geq 2$. (Indeed, the optimal constant here is much smaller than $1/2$.) We deduce that if $N\geq 99M$ then
\begin{equation}
\bE_{\beta,\sigma} \left[|K| \wedge N\right] \geq \frac{1}{2} \bE_\beta |K|
\end{equation}
and hence that
\[
\bE_{\beta,\sigma} \left[|K| \wedge (\lambda M) \right] \geq \frac{\lambda}{100} \bE_{\beta,\sigma} \left[|K| \wedge (100 M) \right] \geq \frac{\lambda}{100} \bE_{\beta,\sigma} \left[|K| \wedge \lceil 99 M\rceil \right] \geq \frac{\lambda}{200} \bE_\beta |K|
\]
for every $0\leq \lambda \leq 1$ as claimed.
\end{proof}

\begin{proof}[Proof of \cref{lem:susceptibility_renormalization}]
Let $L\geq L_0$ where $L_0\geq 2$ is as in \cref{prop:maximum_upper}. 
Write $h=c\beta L^{-(d+\alpha)(n+1)}$, so that each two distinct vertices of $\sigma(B)$ are connected by an edge of $\omega_{\sigma(B)}$ with probability $1-e^{-h}$. Let $B$ be an ancestrally good $n$-block for some $n\geq 0$. Since $B$ is good there exists a sibling $B'$ of $B$ such that
\[\sum_{x\in B'} \bE_{\beta,\sigma}|K_{B'}(x)\cap B'| \geq \sum_{x\in B} \bE_{\beta,\sigma}|K_{B}(x)\cap B|. \]
Let $\cF$ be the sigma-algebra generated by $\eta_B$, which we recall is equal to $\eta_{B'}$, and consider the random collections of sets
\[
\sC = \{ K \cap B : K \text{ is a cluster of $\eta_B$}\} \qquad \text{ and } \qquad \sC' = \{ K \cap B' : K \text{ is a cluster of $\eta_{B'}$}\},
\]
each of which are $\cF$-measurable and satisfy
\[
 \sum_{x\in B} |K_{B}(x)\cap B| = \sum_{C \in \sC} |C|^2  \qquad \text{ and } \qquad \sum_{x\in B'} |K_{B'}(x)\cap B'| = \sum_{C' \in \sC'} |C'|^2.
\]
For each $x\in B$ write $C(x)=K_B(x) \cap B$. For each $x\in B$, conditional on $\cF$, each set $C'\in \sC'$ is connected to $C(x)$ by an edge of $\omega_{\sigma(B)}$ with probability $1-\exp(-h |C'| \cdot |C(x)|)$ so that
\begin{align}
\bE_{\beta,\sigma}\left[|K_{\sigma(B)}(x) \cap \sigma(B)| \mid \cF\right] &\geq \sum_{C' \in \sC'} |C'| \left[1-\exp\left(-h|C'| \cdot |C(x)|\right)\right]
\nonumber\\
&\geq \sum_{C' \in \sC'} |C'| \left[1-\exp\left(-h\left(\frac{1}{\sqrt{h}}\wedge |C'|\right)\left(\frac{1}{\sqrt{h}} \wedge |C(x)|\right)\right)\right]
\end{align}
where the final inequality follows trivially from the fact that $e^t$ is an increasing function of $t$.
Using that  $1-e^{-t} \geq (1-e^{-1}) t$ for $0\leq t\leq 1$ it follows that
\begin{align}
\bE_{\beta,\sigma}\left[|K_{\sigma(B)}(x) \cap \sigma(B)| \mid \cF\right] 
&\geq (1-e^{-1})h  \sum_{C' \in \mathscr{C}'} |C'| \left(\frac{1}{\sqrt{h}} \wedge |C'|\right)\left(\frac{1}{\sqrt{h}} \wedge |C(x)|\right).
\end{align}
Noting that $\sum_{C' \in \mathscr{C}'} |C'| (h^{-1/2} \wedge |C'|)$ is an increasing function of the percolation configuration $\eta_{B'}=\eta_B$, we may apply the Harris-FKG inequality to take expectations and deduce that
%
\begin{align*}
\bE_{\beta,\sigma}|K_{\sigma(B)}(x) \cap \sigma(B)| 
&\geq  (1-e^{-1})h \bE_{\beta,\sigma}\left[\frac{1}{\sqrt{h}} \wedge |C(x)|\right] \bE_{\beta,\sigma}\left[\sum_{C' \in \sC'} |C'| \left(\frac{1}{\sqrt{h}} \wedge |C'|\right)\right]\\
&= (1-e^{-1})h \bE_{\beta,\sigma}\left[\frac{1}{\sqrt{h}} \wedge |C(x)|\right] \sum_{y\in B'}\bE_{\beta,\sigma}\left[\frac{1}{\sqrt{h}} \wedge |K_{B'}(y)\cap B'|\right]
\end{align*}
for every $x\in B$ and hence that
\begin{multline*}
\sum_{x\in B} \bE_{\beta,\sigma}|K_{\sigma(B)}(x) \cap \sigma(B)| 
\\\geq (1-e^{-1})h \sum_{x\in B}  \bE_{\beta,\sigma}\left[\frac{1}{\sqrt{h}} \wedge |K_{B}(x)\cap B|\right] \sum_{y\in B'}\bE_{\beta,\sigma}\left[\frac{1}{\sqrt{h}} \wedge |K_{B'}(y)\cap B'|\right].
\end{multline*}
Since $B$ is ancestrally good, $\sigma(B)$ is ancestrally good also. Thus, since $L\geq L_0$ and $M_{B'}\leq M_{\sigma(B)}$, we have by \cref{prop:maximum_upper}  that there exists a constant $A=A(d,L,\alpha)$ such that $M_B$ and $M_{B'}$ are both at most $\sqrt{\frac{A}{c\beta} L^{(d+\alpha)(n+1)}}$. Since this upper bound is of the same order as $1/\sqrt{h}$, it follows from this and \cref{lem:truncated_susceptibility} that there exists a positive constant $a=a(d,L,\alpha)$ such that
\begin{align*}
\sum_{x\in B} \bE_{\beta,\sigma}|K_{\sigma(B)}(x) \cap \sigma(B)| 
&\geq a h \sum_{x\in B}  \bE_{\beta,\sigma} |K_{B}(x)\cap B| \sum_{y\in B'}\bE_{\beta,\sigma} |K_{B'}(y)\cap B'|\\
&
\geq a h \left(\sum_{x\in B}  \bE_{\beta,\sigma} |K_{B}(x)\cap B|\right)^2,
\end{align*}
where the second inequality follows by choice of $B'$. This implies the claim.
\end{proof}

We next deduce \cref{prop:susceptibility_upper} from \cref{lem:maximum_renormalization}.

\begin{proof}[Proof of \cref{prop:susceptibility_upper}]
Let $a$ and $L_0$ be as in  \cref{lem:susceptibility_renormalization}, let $\beta<\beta_c$ and suppose that $L \geq L_0$. We will prove the claim with $A=a^{-1}L^{2(d+\alpha)}$. Fix $\sigma\in \Sigma$, let $B=B_n$ be an ancestrally good $n$-block for some $n\geq 0$, and for each $m\geq n$ let $B_m$ be the unique $m$-block that is an ancestor of $B_n$. Suppose for contradiction that 
\[
\sum_{x\in B_n}  \bE_{\beta,\sigma} |K_{B_n}(x)\cap B_n| \geq \frac{A}{c\beta} L^{(d+\alpha)n}.
\]
Applying \cref{lem:susceptibility_renormalization} inductively yields  that
\begin{align}
\sum_{x\in B_{m}}  \bE_{\beta,\sigma} |K_{B_m}(x)\cap B_m| &\geq ac \beta L^{-(d+\alpha)m} \left(\sum_{x\in B_{m-1}}  \bE_{\beta,\sigma} |K_{B_{m-1}}(x)\cap B_{m-1}|\right)^2
\nonumber
\\
&
\geq \frac{a c \beta A^2}{c^2 \beta^2 L^{2(d+\alpha)}} L^{(d+\alpha)m} = \frac{A}{c\beta}L^{(d+\alpha)m}
\label{eq:susceptibility_contradiction1}
\end{align}
for every $m\geq n$. On the other hand, we can also use translation-invariance of $J$ to bound
\begin{equation}
\sum_{x\in B_{m}}  \bE_{\beta,\sigma} |K_{B_m}(x)\cap B_m| \leq L^{dm}\bE_\beta |K(0)|
\label{eq:susceptibility_contradiction2}
\end{equation}
where $K(0)$ is the entire cluster of the origin in the full percolation configuration. Since $J$ is translation-invariant and $\beta<\beta_c$, we have by sharpness of the phase transition \cite{duminil2015new,aizenman1987sharpness,1901.10363} that $\bE_\beta |K_0|<\infty$ and hence that the two estimates \eqref{eq:susceptibility_contradiction1} and \eqref{eq:susceptibility_contradiction2} contradict each other for large values of $m$.
\end{proof}

\subsection{Proof of the main theorem}

In this section we deduce \cref{thm:main} from \cref{prop:susceptibility_upper}.
The first step of the proof is the same as in the deduction of \cite[Proposition 2.10]{hutchcrofthierarchical} from \cite[Proposition 2.7]{hutchcrofthierarchical}, and is where we benefit most from the precise way we set up our renormalization scheme.

\begin{lemma}
\label{lem:fullspace_recursion}
Let $\sigma \in \Sigma$, let $B_n$ be an $n$-block for some $n\geq 0$, and for each $m> n$ let $B_m$ be the unique $m$-block that is an ancestor of $B_n$. Then
\begin{multline*}
\bP_{\beta}\left(x\leftrightarrow y\right) \leq \bP_{\beta}\left(x\leftrightarrow y \text{ in $\eta_{B_n}$}\right)
\\ + c\beta \sum_{m=n+1}^\infty L^{-(d+\alpha)m}\sum_{a,b\in B_m} \bP_{\beta,\sigma}(x \leftrightarrow a \text{ in $\eta_{B_m}$}) \bP_{\beta,\sigma}(y \leftrightarrow b \text{ in $\eta_{B_m}$}).
\end{multline*}
\end{lemma}

The proof of this lemma will employ the \emph{BK inequality} and the related notion of the \emph{disjoint occurrence} of events; see e.g.\ \cite[Chapter 2.3]{grimmett2010percolation} for relevant background.

\begin{proof}[Proof of \cref{lem:fullspace_recursion}]
For each $m\geq n$ write $\eta_m=\eta_{B_m}$ and $K_m(x)=K_{B_m}(x)$, so that we can write $\{x\leftrightarrow y$ in $\eta_m\}=\{y\in K_m(x)\}$.
Since $\omega = \bigcup_{m\geq n} \eta_m$ we can write
\[
\bP_{\beta,\sigma}(y\in K(x)\setminus K_n(x)) = \sum_{m=n+1}^\infty \bP_{\beta,\sigma}\left(y\in K_{m}(x)\setminus K_{{m-1}}(x)\right).
\]
If $x$ is connected to $y$ in $\eta_{m}$ but not $\eta_{{m-1}}$ then every path connecting $x$ and $y$ in $\eta_m$ must include an edge of $\omega_{B_m}$. As such, on this event there must exist a pair of points $a,b \in B_m$ such that the events $\{\{a,b\}$ is open in $\omega_{B_m}\}$, $\{a\in K_{m}(x)\}$, and $\{b\in K_{m}(y)\}$  all occur disjointly. Applying a union bound and the BK inequality and using that any two vertices of $B_m$ are connected by an edge of $\omega_{B_m}$ with probability at most $c \beta L^{-(d+\alpha)m}$ yields the claim.
\end{proof}

If we were working on the hierarchical lattice, the symmetries of the model would allow us to very easily reach our desired conclusions on the unrestricted two-point function from \cref{prop:susceptibility_upper} and \cref{lem:fullspace_recursion} by direct summation, as we did in \cite{hutchcrofthierarchical}. In our setting however things are rather more subtle since including the hierarchical structure breaks transitivity and \cref{prop:susceptibility_upper} does not obviously give sharp control of sums  of the form $\sum_{x,y\in B_n}\sum_{a,b\in B_m} \bP_{\beta,\sigma}(x \leftrightarrow a \text{ in $\eta_{B_m}$}) \bP_{\beta,\sigma}(y \leftrightarrow b \text{ in $\eta_{B_m}$})$ even when $B_n$ is ancestrally good. In order to circumvent this issue we will use the fact that every ancestrally good block has many ancestrally good descendants to establish via a further compactness argument the existence of a hierarchical decomposition $\sigma$ where these sums can be controlled.

\medskip

 Given $\sigma\in \Sigma$, $\beta\geq 0$, and an $n$-block $B_n \in \cB^\sigma_n$ for some $n\geq 0$, we write $B_{n+k}$ for the unique $(n+k)$-block containing $B_n$ and define
\[
\mathbf{T}_k(B_n) =  \mathbf{T}_{k,\beta,\sigma}(B_n) = \sum_{x\in B_n} \sum_{y\in B_{n+k}} \bP_{\beta,\sigma} (x \leftrightarrow y 
\text{ in $\eta_{B_{n+k}}$}) \qquad \text{ for each $k\geq 0$},
\]
 so that summing the estimate of \cref{lem:fullspace_recursion} over $x,y\in B_n$ yields in this notation that
\begin{equation}
\label{eq:Rk_expansion}
\sum_{x,y\in B_n}\bP_{\beta}\left(x\leftrightarrow y\right) \leq \mathbf{T}_0(B_n) + c\beta L^{-(d+\alpha)n} \sum_{k=1}^\infty L^{-(d+\alpha)k}\mathbf{T}_k(B_n)^2
\end{equation}
for every $n$-block $B_n$. The first term is bounded directly by \cref{prop:susceptibility_upper}. The following lemma applies \cref{prop:susceptibility_upper} to give a non-sharp bound of reasonable order on $\mathbf{T}_k(B_n)$ for certain well-chosen $n$-blocks.


\begin{lemma}
\label{lem:excellent_block}
There exists an integer $L_0=L_0(d,\alpha) \geq 2$ such that the following holds: If $L \geq L_0$ then
there exists a constant $A=A(d,L,\alpha)$ such that if $\sigma\in \Sigma$ and $B_{n+\ell}$ is an ancestrally good $(n+\ell)$-block for some $n,\ell\geq 0$, then  there exists an ancestrally good $n$-block $B_n$ descended from $B_{n+\ell}$ such that
\[
\mathbf{T}_k(B_n) \leq \frac{A}{c\beta} 4^k L^{(d+\alpha)n + \alpha k}
\]
for every $0\leq k \leq \ell$.
\end{lemma}

Note that, once \cref{thm:main} is established, we will deduce \emph{a posteriori} that the same estimate holds without the $4^k$ term.
When applying this lemma we will take $L$ sufficiently large that $4^k\leq L^{\delta k}$ for an appropriately small $\delta>0$.

\begin{proof}[Proof of \cref{lem:excellent_block}]
Let $\sigma\in \Sigma$, let $n,\ell \geq 0$, and let $B_{n+\ell}$ be an ancestrally good $(n+\ell)$-block. 
 For each $0\leq k \leq \ell$ let $\sA_k$ be the collection of all $(n+k)$-blocks descended from $B_{n+\ell}$ that are ancestrally good and let $\sB_k$ be the collection of all $n$-blocks descended from blocks of $\sA_k$, so that $\sB_\ell$ is the collection of all $n$-blocks descended from $B_{n+\ell}$ and $\sB_0=\sA_0$ is the collection of all $n$-blocks descended from $B_{n+\ell}$ that are ancestrally good.
For each $0\leq k \leq \ell$  we have by the definitions that
\begin{equation}
\label{eq:BkAk_identity}
\sum_{B \in \sB_{k}} \mathbf{T}_k(B) = \sum_{B \in \sA_k} \mathbf{T}_0(B).
\end{equation}
Meanwhile, since every non-singleton block has at least $L^d/2$ children that are good, every block in $\sA_k$ is an ancestor of at least $2^{-k}L^{dk}$ blocks in $\sB_0$. Since $\sB_0 \subseteq \sB_k$ it follows that
\begin{equation}
\frac{1}{|\sB_0|} \sum_{B \in \sB_{0}} \mathbf{T}_k(B) \leq \frac{|\sB_k|}{|\sB_0|} \cdot \frac{1}{|\sB_k|} \sum_{B \in \sB_{k}}\mathbf{T}_k(B) \leq \frac{2^k}{|\sB_k|} \sum_{B \in \sB_{k}}\mathbf{T}_k(B)
\end{equation}
for each $1\leq k \leq \ell$. 
Letting $\sD_k=\{B\in \sB_0 : \mathbf{T}_k(B) \geq \frac{4^k}{|\sB_k|} \sum_{B \in \sB_{k}}\mathbf{T}_k(B)\}$, it follows that
\begin{align}
|\sD_k| &\leq \left( \frac{4^k}{|\sB_k|} \sum_{B \in \sB_{k}}\mathbf{T}_k(B) \right)^{-1} \sum_{B \in \sD_k} \mathbf{T}_k(B)
\nonumber\\
&\leq  \left( \frac{4^k}{|\sB_k|} \sum_{B \in \sB_{k}}\mathbf{T}_k(B) \right)^{-1} \sum_{B \in \sB_0} \mathbf{T}_k(B)
%
 \leq 2^{-k} |\sB_0|
\end{align}
for every $1\leq k \leq \ell$. Since $\sum_{k=1}^\ell 2^{-k}<1$, we deduce that $\bigcup_{k=1}^\ell \sD_k$ is a strict subset of $\sB_0$ and hence that there  exists $B_n\in \sB_0$ such that
\begin{equation}
\mathbf{T}_k(B_n) \leq \frac{4^k}{|\sB_k|} \sum_{B \in \sB_{k}}\mathbf{T}_k(B)
\end{equation}
for every $1 \leq k \leq \ell$.as Since $|\sB_k|=L^{dk}|\sA_k|$, using the identity \eqref{eq:BkAk_identity} and applying \cref{prop:susceptibility_upper} to bound $\mathbf{T}_0(B)$ for each $B \in \sA_k$ yields that there exist constants $L_0=L_0(d,\alpha)$ and $A=A(d,L,\alpha)$ such that if $L\geq L_0$ then this block $B_n$ satisfies
\begin{equation}
\mathbf{T}_k(B_n) \leq 4^k L^{-dk} \frac{A}{c\beta} L^{(d+\alpha)(n+k)}
\end{equation}
for every $1\leq k \leq \ell$ as claimed. The same estimate also holds with $k=0$ by direct application of \cref{prop:susceptibility_upper}.
%
%
\end{proof}

\begin{corollary}
\label{cor:excellent_block}
There exists an integer $L_0=L_0(d,\alpha) \geq 2$ such that the following holds: If $L \geq L_0$ then
there exists a constant $A=A(d,L,\alpha)$ such that for each $n\geq 0$ and $0<\beta <\beta_c$ there exists $\sigma \in \Sigma$ such that the block $B_n=B^\sigma_n(0)$ satisfies
\[
\mathbf{T}_k(B_n) \leq \frac{A}{c\beta} 4^k L^{(d+\alpha)n+\alpha k}
\]
for every $k\geq 0$.
\end{corollary}

\begin{proof}[Proof of \cref{cor:excellent_block}]
Fix $0<\beta<\beta_c$ and $n\geq0$, let $L_0$ be as in \cref{lem:excellent_block} and suppose that $L \geq L_0$. 
By \cref{lem:ancestrally_minimal} and \cref{lem:excellent_block} we can find for each  $\ell \geq 0$ an $L$-adic decomposition $\sigma$ of $\Z^d$ and an $n$-block $B_n$ of $\sigma$ such that
\[
\mathbf{T}_k(B_n) = \mathbf{T}_{k,\beta,\sigma}(B_n) \leq \frac{A}{c\beta} 4^k L^{(d+\alpha)n+\alpha k}
\]
for every $0\leq k \leq \ell$. By translating $\sigma$ we may take this block $B_n$ to be the $n$-block containing the origin. It follows from the proof of \cref{lem:weak_continuity} that $\mathbf{T}_{k,\beta,\sigma}(B_n^\sigma)$ is a continuous function of $\sigma \in \Sigma$, so that taking a subsequential limit of these decompositions as $\ell\to\infty$ yields a decomposition with the desired properties.
\end{proof}

\begin{proof}[Proof of \cref{thm:main}]
Fix $0<\beta <\beta_c$ and $n\geq 1$, let $L_0$ be as in \cref{cor:excellent_block} and take $L =  L_0 \vee \lceil (2^{5})^{1/(d-\alpha)}\rceil$. Taking the $L$-adic decomposition $\sigma$ whose existence is guaranteed by \cref{cor:excellent_block} and letting $B_n=B_n^\sigma$, we have by that corollary and by \eqref{eq:Rk_expansion} that there exists a constant $A_1=A_1(d,\alpha)$ such that
\begin{align*}\sum_{x,y\in B_n}\bP_{\beta}\left(x\leftrightarrow y\right) &\leq \mathbf{T}_0(B_n) + c\beta L^{-(d+\alpha)n} \sum_{k=1}^\infty L^{-(d+\alpha)k}\mathbf{T}_k(B_n)^2\\
&\leq \frac{A_1}{c\beta}L^{(d+\alpha)n} + c\beta L^{-(d+\alpha)n} \sum_{k=1}^\infty L^{-(d+\alpha)k} \frac{A_1^2}{(c\beta)^2} L^{2(d+\alpha)n+2\alpha k}
\\
&= \frac{A_1}{c\beta}L^{(d+\alpha)n} + \frac{A_1^2}{c\beta} L^{(d+\alpha)n} \sum_{k=1}^\infty 2^{4k} L^{-(d-\alpha)k} \leq \frac{A_1 + A_1^2}{c\beta} L^{(d+\alpha)n},
\end{align*}
where we used that $L^{d-\alpha} \geq 2^5$ in the final inequality. Next observe that if $x$ belongs to one of the $(L-2)^d$ children of $B_n$ that does not intersect the boundary of $B_n$ then the box $x+\Lambda_{L^{n-1}}$ is contained in the block $B_n$. Since there are $(L-2)^d L^{d(n-1)}$ many such $x$, it follows by transitivity that
\[\sum_{x,y\in B_n}\bP_{\beta}\left(x\leftrightarrow y\right) \geq (L-2)^d L^{d(n-1)} \sum_{a \in \Lambda_{L^{n-1}} } \bP_\beta(0\leftrightarrow a).
\]
Combining these two estimates yields that
\[
\sum_{a \in \Lambda_{L^{n-1}} } \bP_\beta(0\leftrightarrow a) \leq \left(\frac{L}{L-2}\right)^d\frac{A_1 +  A_1^2}{c\beta} L^{\alpha n}.
\]
Since $n\geq 0$ and $0<\beta<\beta_c$ were arbitrary and $A_1$ and $L$ were chosen as functions of $d$, $\alpha$, and $c$, and since every $r\geq 1$ is within a factor of $L$ of a power of $L$, it follows that there exists a constant $A_2=A_2(d,\alpha)$ such that
\[
\sum_{a \in \Lambda_r} \bP_\beta(0\leftrightarrow a) \leq \frac{A_2}{c\beta} r^{\alpha}
\]
for every $0<\beta <\beta_c$ and $r\geq 1$. Since connection probabilities can be written as suprema of connection probabilities in finite boxes, they are left-continuous in $\beta$ and the same estimate must also hold at $\beta_c$. 
\end{proof}

\section{Corollaries for the tail of the volume}

We now deduce \cref{cor:volume} from \cref{thm:main}. The proof will apply the following strong form of the \emph{two-ghost inequality} proven in \cite[Theorem 3.1]{hutchcroft2020power}, which we state in the special case of long-range percolation on $\Z^d$. Given a translation-invariant kernel on $\Z^d$ we write $J_x=J(0,x)$ for every $x\in \Z^d$ and write $\sS'_{x,n}$ for the event that $0$ and $x$ both belong to distinct clusters each of which contain at least $n$ vertices and at least one of which is finite.

 \begin{theorem}[Two-ghost inequality]
\label{thm:two_ghost_S} Let $J$ be a translation-invariant kernel on $\Z^d$, let $\beta \geq 0$, 
and suppose that there exist constants $A<\infty$ and $0\leq \theta <1/2$ such that $\bP_{\beta}(|K_o| \geq n) \leq A n^{-\theta}$ for every $n \geq 1$. Then
\begin{align}
\label{eq:improved_two_ghost}
\sum_{x \in \Z^d} (e^{\beta J_x}-1) \bP_{\beta}(\sS_{x,n}')^2 \leq \frac{40000 \cdot A^2}{(1-2\theta)^2 n^{1+2\theta}}  \qquad \text{ for every $n\geq 1$.}
\end{align}
\end{theorem}

This inequality strengthens inequalities appearing in our earlier works \cite{1808.08940,Hutchcroft2020Ising} and is closely related to the classical work of Aizenman, Kesten, and Newman \cite{MR901151}.

\begin{proof}[Proof of \cref{cor:volume}]
 Let $\theta=(d-\alpha)/2d<1/2$. We first claim that there exists a constant $C_0=C_0(d,\alpha)$ such that the implication
\begin{multline}\text{$\Bigl(\bP_\beta(|K|\geq n) \leq A n^{-\theta}$ for every $n\geq 1\Bigr)$}\\ \Rightarrow  \text{$\Bigl(\bP_\beta(|K|\geq n) \leq \frac{C_0 \sqrt{A+1}}{(c\beta)^{d/(4d-2\alpha)}} n^{-\theta}$ for every $n\geq 1\Bigr)$}
\label{eq:general_bootstrap}
\end{multline}
holds for every $1 \leq A < \infty$ and $0 \leq \beta < \beta_c$. 
 
 Fix $0< \beta <\beta_c$ and $1 \leq A < \infty$ and suppose that $\bP_\beta(|K|\geq n) \leq A n^{-\theta}$ for every $n\geq 1$.
We have by translation-invariance, the Harris-FKG inequality, and a union bound that
\begin{equation}
\label{eq:botstrap_union}
\bP_\beta(|K|\geq n)^2 \leq \bP_\beta(|K(x)|\geq n \text{ and } |K(0)|\geq n) \leq \bP_\beta(\sS'_{x,n})+\bP_\beta(0\leftrightarrow x)
\end{equation}
for each $x\in \Z^d$ and $n\geq 1$.
Since $\theta<1/2$, we have by \cref{thm:two_ghost_S} that
there exists a constant $C_1=C_1(d,\alpha)$ such that
\[
\sum_{x \in \Z^d} (e^{\beta J_x}-1) \bP_\beta(\sS_{x,n}')^2 \leq C_1^2A^2 n^{-(1+2\theta)}
\]
for every $n \geq 1$, and it follows by Cauchy-Schwarz that there exists a constant $C_2=C_2(d,\alpha)$ such that
\begin{align}
\hspace{-0.5em}\sum_{x \in \Lambda_r} \bP_\beta(\sS_{x,n}') &\leq \left[\sum_{x \in \Lambda_r} (e^{\beta J_x}-1) \bP_\beta(\sS_{x,n}')^2\right]^{1/2}\left[\sum_{x \in \Lambda_r \setminus \{0\}} \frac{1}{e^{\beta J_x}-1}\right]^{1/2} 
\nonumber
\\
&\leq C_1 A n^{-(1+2\theta)/2} \left( \frac{1}{c\beta r^{-d-\alpha}} |\Lambda_r| \right)^{1/2} 
\leq \frac{C_2 A}{\sqrt{c\beta}} n^{-(1+2\theta)/2} r^{d+\alpha/2}
\label{eq:two_ghost_massage}
\end{align}
for every $r\geq 1$, where we used that $e^x-1 \geq x$ in the  second line. Averaging \eqref{eq:botstrap_union} over $x\in \Lambda_r$ and using \eqref{eq:two_ghost_massage} to control the first term and \cref{thm:main} to control the second yields that there exists a constant $C_3=C_3(d,\alpha)$ such that
\begin{equation}
\label{eq:about_to_optimize}
\bP_\beta(|K|\geq n)^2 \leq  \frac{C_2 A}{\sqrt{c\beta}
} n^{-(1+2\theta)/2} r^{\alpha/2} + \frac{C_3}{c\beta} r^{-d+\alpha}
\end{equation}
for every $n,r\geq 1$. We optimize this bound by taking $r^{d-\alpha/2}=\lceil n^{(1+2\theta)/2} (c\beta)^{-1/2}\rceil$, noting that \cref{cor:betabound} implies that there exists a positive constant $\beta_1=\beta_1(\alpha,d)$ such that $c\beta \leq \beta_1$ and hence that $\lceil n^{(1+2\theta)/2} (c\beta)^{-1/2}\rceil$ is bounded above by $C_4 n^{(1+2\theta)/2} (c\beta)^{-1/2}$ for some $C_4=C_4(d,\alpha)$. (Here we are just using that rounding up a number that is bounded away from zero increases that number by at most a bounded multiplicative factor.) Substituting this value of $r$ into \eqref{eq:about_to_optimize} yields that there exists a constant $C_5=C_5(d,\alpha)$ such that
\begin{equation}
\bP_\beta(|K|\geq n)^2 \leq C_5( A + 1) (c\beta)^{-d/(2d-\alpha)} n^{-(1+2\theta)\frac{d-\alpha}{2d-\alpha}}= C_5( A + 1) (c\beta)^{-d/(2d-\alpha)} n^{-2\theta}
\end{equation}
where the final equality follows by choice of $\theta$. Taking square roots of both sides concludes the proof of \eqref{eq:general_bootstrap}.

We now deduce the claimed inequality from the bootstrap implication \eqref{eq:general_bootstrap}. For each $0<\beta<\beta_c$ consider the quantity
\[
A(\beta)=\min\Bigl\{ A \geq 1 : \bP_\beta(|K|\geq n) \leq An^{-\theta} \text{ for every $n\geq 1$}\Bigr\},
\]
which is finite by sharpness of the phase transition. (Indeed, for $\beta<\beta_c$ the tail probability $\bP_\beta(|K|\geq n)$ decays exponentially in $n$ \cite{1901.10363}.) Applying \eqref{eq:general_bootstrap} yields that
\[
A(\beta) \leq \frac{C_0\sqrt{A(\beta)+1}}{(c\beta)^{d/(4d-2\alpha)}}
\]
for every $\beta < \beta_c$, and hence by \cref{cor:betabound} that there exists a constant $C_6=C_6(d,\alpha)$ such that $A(\beta) \leq C_6 (c\beta)^{-d/(2d-\alpha)}$ for every $0< \beta <\beta_c$. It follows in particular that 
\[\bP_\beta(|K|\geq n) \leq C_6 (c\beta)^{-d/(2d-\alpha)} n^{-\theta}\] for every $0<\beta<\beta_c$ and $n\geq 1$, and hence also for $\beta=\beta_c$ by monotone convergence.
\end{proof}

\section{Open problems}

We close the paper by highlighting some interesting directions for future research that we believe may be within the scope of modern methods.

\medskip

\noindent \textbf{Lower bounds.} Perhaps the most obvious question raised  by this work is as follows. The problem is most interesting when $\alpha \geq d/3$ and $d \leq 6$ so that high-dimensional techniques should not apply.

\begin{problem}
Find conditions under which the upper bound of \cref{thm:main} admits a matching lower bound.
\end{problem}

\begin{remark}\label{remark:BaumlerBerger}
After this paper first appeared, B\"aumler and Berger \cite{baumler2022isoperimetric} established a matching lower bound holding for all $\alpha < 1$, for every dimension $d\geq 1$. (A lower bound that is matching to within a log factor also holds for $\alpha =1$.) Together with our results, this establishes that the exponent $\eta$ satisfies $2-\eta=\alpha$ whenever it is well-defined and $\alpha \leq 1$, which handles all relevant values of $\alpha$ in the one-dimensional case as well as a non-trivial interval $\alpha \in (2/3,1)$ of non-mean-field models in the two-dimensional case.
\end{remark}

\noindent \textbf{Non-perturbative proofs of mean-field critical behaviour}.
Regarding the high-dimensional case, we were frustrated that we were not able to prove a strong enough form of \cref{thm:main} to prove that the model always has mean-field critical behaviour for $\alpha<d/3$; there is too much averaging in our bounds to use them to bound the triangle diagram directly. 

\begin{problem}
\label{problem:triangle}
Prove under the hypotheses of \cref{thm:main} that the triangle condition is satisfied when $\alpha<d/3$.
\end{problem}

See \cite{HutchcroftTriangle,heydenreich2015progress} for background on the triangle condition. Note that mean-field behaviour of the model in this regime has already been established under perturbative conditions using the lace expansion \cite{MR2430773,MR3306002}; the problem is to give a non-perturbative proof. 
It would also be very interesting to prove that mean-field critical behaviour holds to within polylogarithmic factors when $\alpha=d/3$ as was done for the hierarchical lattice in \cite{hutchcrofthierarchical,HutchcroftTriangle}.

\medskip

\noindent \textbf{\cref{thm:main} with less averaging.}
There are two natural directions to try to strengthen \cref{thm:main}, either of which may lead to a solution of \Cref{problem:triangle} (perhaps under slightly stronger hypotheses): pointwise bounds and bounds on Fourier coefficients. 

\begin{problem}
Prove under the hypotheses of \cref{thm:main} that if $J$ also satisfies an upper bound of the form $J(x,y)\leq C \|x-y\|^{-d-\alpha}$ then $\bP_{\beta_c}(x\leftrightarrow y) \preceq \|x-y\|^{-d+\alpha}$ for $x,y\in \Z^d$ distinct.
\end{problem}

\begin{problem}
Prove under the hypotheses of \cref{thm:main} that the Fourier transform $\hat \tau_{\beta_c}$ of the two-point function $\tau_{\beta_c}(x)=\bP_{\beta_c}(0\leftrightarrow x)$ satisfies the bound $\hat \tau_{\beta_c}(\theta) \preceq \|\theta\|^{-\alpha}$ for every $\theta \in [-\pi,\pi]^d$.
\end{problem}

 \noindent \textbf{Near-critical behaviour.} In a more speculative direction, let us mention the problem of understanding the near-critical behaviour of the model. The interested reader may find the work of Slade on the spin $O(n)$ model \cite{MR3772040} to be inspiring.

\begin{problem}
Investigate the near-critical behaviour of the model. Are the exponents $\gamma$ and $\beta$ the same for long-range percolation on $\Z^d$ and the hierarchical lattice when $\alpha<\alpha_c$?
\end{problem}

Let us end with the following vague question, referring the reader to \cite{LimitsOfUniversality} for related discussions.

\begin{question}
To what extent can long-range percolation on $\Z^d$ and the hierarchical lattice be said to belong to ``the same universality class" when $\alpha<\alpha_c$?
\end{question}

\subsection*{Acknowledgments} We thank Philip Easo, Emmanuel Michta, Gordon Slade, and the anonymous referee for helpful comments on earlier versions of this manuscript.






 \setstretch{1}
 \footnotesize{
  \bibliographystyle{abbrv}
  \bibliography{unimodularthesis.bib}
  }

\end{document}